\documentclass[reqno]{amsart}
\usepackage{amssymb}
\usepackage{graphicx}
\usepackage{enumerate}
\usepackage[usenames, dvipsnames]{color}
\usepackage{verbatim}
\usepackage{subfigure}
\usepackage{mathrsfs}
\usepackage{bm}
\usepackage{cite}

\allowdisplaybreaks
\numberwithin{equation}{section}

\newtheorem{theorem}{Theorem}[section]

\newtheorem{lemma}[theorem]{Lemma}
\newtheorem{prop}[theorem]{Proposition}

\theoremstyle{definition}
\newtheorem{remark}[theorem]{Remark}

\theoremstyle{definition}

\theoremstyle{definition}

\theoremstyle{definition}
\newtheorem{example}[theorem]{Example}

\makeatletter
\def\dashint{\operatorname%
{\,\,\text{\bf-}\kern-.98em\DOTSI\intop\ilimits@\!\!}}
\makeatother

\def\\det{\text{det}}

\def\.5{\frac{1}{2}}

\newcommand{\RN}[1]{%
  \textup{\uppercase\expandafter{\romannumeral#1}}%
}

\renewcommand{\epsilon}{\varepsilon}

\newcounter{marnote}

\begin{document}


\title[Characterization of Electric Fields ]{Characterization of Electric Fields for Perfect Conductivity Problems in 3D}

\author[H.G. Li]{Haigang Li}
\address[H.G. Li]{School of Mathematical Sciences, Beijing Normal University, Laboratory of Mathematics and Complex Systems, Ministry of Education, Beijing 100875, China. }
\email{hgli@bnu.edu.cn}
\thanks{H.G. Li was partially supported by  NSFC (11571042) (11631002), Fok Ying Tung Education Foundation (151003). }

\author[F. Wang]{Fang Wang}
\address[F. Wang]{School of Mathematical Sciences, Beijing Normal University, Laboratory of Mathematics and Complex Systems, Ministry of Education, Beijing 100875, China.}
\email{wf@mail.bnu.edu.cn}

\author[L.J. Xu]{Longjuan Xu}
\address[L.J. Xu]{School of Mathematical Sciences, Beijing Normal University, Laboratory of Mathematics and Complex Systems, Ministry of Education, Beijing 100875, China.}
\email{ljxu@mail.bnu.edu.cn. Corresponding author.}

\maketitle

\begin{abstract}
In composite materials, the inclusions are frequently spaced very closely. The electric field concentrated in the narrow regions between two adjacent perfectly conducting inclusions will always become arbitrarily large. In this paper, we establish an asymptotic formula of the electric field in the zone between two spherical inclusions with different radii in three dimensions. An explicit blowup factor relying on radii is obtained, which also involves the digamma function and Euler-Mascheroni constant, and so the role of inclusions' radii played in such blowup analysis is identified.
\end{abstract}

\section{Introduction and main results}

In this paper, we investigate the blowup phenomena that occur in composite materials consisting of a finite conductivity matrix and perfectly conducting inhomogeneities close to touching and derive the asymptotic formula of the electric field in the narrow region between two perfectly conducting inclusions in three dimensions. For two spherical inclusions with different radii, we obtain an explicit blowup factor involving the digamma function and Euler-Mascheroni constant and reveal the role of radii of inclusions played in such blowup analysis.

This problem was initiated by Babuska et al \cite{BASL} in the study of fiber-reinforced composite material, where one has to estimate the magnitude of local fields in the zone of high stress field concentration. It can be modeled by a class of divergence form second-order elliptic equations $\nabla\big(a(x)\nabla u\big)=0$ with piecewise constant coefficients, given by $a(x)=k$ for $x$ inside the inclusions, and $a(x)=1$ in the matrix. This model has attracted considerable attention because it can describe various physical phenomena, including electrical conductivity, thermal conduction, anti-plane elasticity, and even flow in porous media. For the sake of definiteness, this paper uses the electrical conductivity language, where $\nabla u$ describes an electric field.

There have been much important progress made on the gradient estimates of the solutions to such elliptic equations $\nabla\big(a(x)\nabla u\big)=0$ since the numerical analysis was studied in \cite{BASL}. For the case when the conductivity stays away from 0 and $\infty$, Bonnetier and Vogelius \cite{BV} first proved that $|\nabla u|$ is bounded for two touching disks $D_{1}$ and $D_{2}$ in two dimensions. Moreover, they pointed out that the bound depends on the value of conductivity. Li and Vogelius \cite{LV} extended the result to general divergence form second-order elliptic equations with piecewise smooth coefficients and they proved that $|\nabla u|$ remains uniformly bounded with respect to the distance between inclusions of arbitrary smooth shape in all dimensions. Li and Nirenberg \cite{LN} extended this result to elliptic systems including systems of linear elasticity, which is assumed in \cite{BASL}. The estimates in \cite{LN} and \cite{LV} depend on the ellipticity of the coefficients. If ellipticity constants are allowed to be deteriorate, the situation is very different. It was shown in various papers, see for example Budiansky and Carrier \cite{BC} and Markenscoff \cite{M}, that when $k=+\infty$ the $L^{\infty}$-norm of $|\nabla u|$ generally becomes unbounded as the distance $\varepsilon$ between inclusions tends to zero. The rate at which the $L^{\infty}$-norm of the gradient of a special solution blows up was shown in \cite{BC} to be $\varepsilon^{-1/2}$ in two dimensions.

In this paper, we consider the perfect conductivity problem, where $k=+\infty$. It was proved by Ammari, Kang and Lim in \cite{AKL} and Ammari et al. in \cite{AKLLL} that when $D_{1}$ and $D_{2}$ are disks in $\mathbb R^{2}$, the blowup rate of $|\nabla u|$ is $\varepsilon^{-1/2}$ as $\varepsilon$ goes to zero; with the lower bound given in \cite{AKL} and the upper bound given in \cite{AKLLL}. This result was extended by Yun \cite{Y1,Y2} and Bao, Li and Yin \cite{BLY1} to strictly convex subdomains $D_{1}$ and $D_{2}$ in $\mathbb R^{2}$. In three dimensions and higher dimensions, the blowup rate of $|\nabla u|$ turns out to be $|\varepsilon \log\varepsilon|^{-1}$ and $\varepsilon^{-1}$, respectively; see \cite{BLY1,LY1}. For related works on elliptic equations and systems arising from the study of composite materials, see \cite{ABTV,ADKL,AGKL,AKLLZ,BLY2,BLL,BLL2,BJL,BT1,BT2,BCN,DL,DZ,GN,K1,K,KY,LL,LLBY,LX,LY1,LY2,MNT,Y3,Y4} and the references therein.

The results mentioned above are estimates of $|\nabla u|$ from above and below, namely,
\begin{equation}\label{lower-upper}
\frac{C_{1}}{\rho(\varepsilon)}\leq|\nabla u|\leq\frac{C_{2}}{\rho(\varepsilon)}+C_{3}
\end{equation}
for some positive constants $C_{1}, C_{2}$ and $C_{3}$, where $\rho(\varepsilon)=
\sqrt{\varepsilon}$, $(\varepsilon|\log\varepsilon|)$, $\varepsilon$,  if $n=2$, $n=3$, $n\geq4$, respectively,
and shows that the electric filed may blow up in the narrow regions between inclusions.

The interest of this paper lies in further establishing the asymptotic formula of $|\nabla u|$ in the narrow zone of electric field concentration. In dimension two, Kang, Lim and Yun \cite{KLY1} obtained a complete characterization of the singular behavior of $\nabla u$ when inclusions are disks. Let $D_{1}$ and $D_{2}$ be disks in $\mathbb R^{2}$ of radii $r_{1}$ and $r_{2}$, respectively, and let $R_{j}$ be the reflection with respect to $\partial D_{i}$, $i=1,2$. Then the combined reflections $R_{1}R_{2}$ and $R_{2}R_{1}$ have unique fixed points, say $\textbf{f}_{1}\in D_{1}$ and $\textbf{f}_{2}\in D_{2}$. Let
\begin{equation}\label{h 2d}
h(\textbf{x})=\frac{1}{2\pi}\big(\log|\textbf{x}-\textbf{f}_{1}|-\log|\textbf{x}-\textbf{f}_{2}|\big).
\end{equation}
It has been proved that the solution $u$ to \eqref{instruction} can be expressed as
\begin{equation}\label{expre u}
u(\textbf{x})=\frac{4\pi r_{1}r_{2}}{r_{1}+r_{2}}(\textbf{n}\cdot\nabla H)(\textbf{c})h(\textbf{x})+g(\textbf{x}),\quad\textbf{x}\in\mathbb R^{2}\setminus(D_{1}\cap D_{2}),
\end{equation}
where $\textbf{c}$ is the middle point of the shortest line segment connecting $\partial D_{1}$ and $\partial D_{2}$, $\textbf{n}$ is the unit vector in the direction of $\textbf{f}_{2}-\textbf{f}_{1}$, and $|\nabla g(\textbf{x})|$ is bounded independently of $\varepsilon$ on any bounded subset of $\mathbb R^{2}\setminus(D_{1}\cap D_{2})$. So the singular behavior of $\nabla u$ is completely characterized by $\nabla h$. Ammari et al. \cite{ACKLY} extended the characterization \eqref{expre u} to the case when inclusions are strictly convex domains in $\mathbb R^{2}$ by using disks osculating to convex domains. It is worth mentioning that stress concentration factor was derived by Gorb in \cite{G2015}, and by Gorb and Novikov  in \cite{GN} for the $p$-Laplacian.

Compared with the known results in dimension two, the situation becomes more complicated in dimension three. Although the singular function $h(\textbf{x})$ can be founded, it is of form of series, see \eqref{def h} below, rather than a function like \eqref{h 2d}. Recently, for a special case that two inclusions with the same radii $r_{1}=r_{2}$, an asymptotic formula of $|\nabla u|$ was obtained by Kang, Lim and Yun in \cite{KLY2}, where the symmetry of the domain makes the computation easy to handle. However, for the general case that $r_{1}\neq r_{2}$, the symmetry is broken and the computation becomes involved. It is not obvious to generalize the asymptotic expression of $|\nabla u|$. In this paper, we mainly overcome this difficulty and obtain a blowup factor making its dependence on the radii explicit, which maybe is useful from the engineering point view. We would like to point out that Lim and Yun \cite{LY1} obtained the upper and lower bounds of $|\nabla u|$ for two balls with different radii by image charge method.  In this paper, we improve that and provide a complete expression of $\nabla u$.

In order to describe the problem and results, we first fix our domains and notations. Let
\begin{equation*}
D_{1}=\mathfrak{B}_{1}=B_{r_{1}}(\textbf{c}_{1}), \quad D_{2}=\mathfrak{B}_{2}=B_{r_{2}}(\textbf{c}_{2})
\end{equation*}
be two balls in $\mathbb R^{3}$, with $2\varepsilon$ apart, where
\begin{equation*}
{\textbf{c}}_{1}=(r_{1}+\varepsilon, 0, 0),\quad \textbf{c}_{2}=(-r_{2}-\varepsilon, 0, 0).
\end{equation*}

Suppose that the conductivity of the inclusions degenerates to $+\infty$; in other words, inclusions are perfect conductors. Consider the following perfect conductivity problem \cite{KLY2}:
\begin{align}\label{instruction}
\begin{cases}
\Delta u=0 &\mbox{in}~~\mathbb{R}^{3}\setminus\overline{D_{1}\bigcup D_{2}},\\
u=C_{i}~(\mbox{constant}) &\mbox{on}~~\partial D_{i}, i=1,2,\\
u(\textbf{x})-H(\textbf{x})=O(|\textbf{x}|^{-2}) &\mbox{as}~~|\textbf{x}|\rightarrow \infty,\\
\int_{\partial D_{i}}\frac{\partial u}{\partial \nu^{i}}\ d\sigma=0 &\mbox{for}~~i=1, 2,
\end{cases}
\end{align}
where $H$ is a given harmonic function in $\mathbb{R}^{3}$ so that $-\nabla H$ is the background electric field in the absence of the inclusions. Here and throughout this paper, $\nu^{i}$ is the outward unit normal vector to $\partial D_{i}$, $i=1,2$.

Let $\rho(\textbf{x})=\sqrt{x_{2}^{2}+x_{3}^{2}}$ and denote
$$r_{\max}:=\max\{r_{1},r_{2}\},\quad r_{\min}:=\min\{r_{1},r_{2}\}.$$
Define the blowup factor
\begin{equation}\label{def psi}
\Psi(r_{1}, r_{2}):=
\frac{\psi\big(\frac{r_{\max}}{r_{1}+r_{2}}\big)C^{H}_{\min}
+\psi\big(\frac{r_{\min}}{r_{1}+r_{2}}\big)C^{H}_{\max}}
{\psi\big(\frac{r_{2}}{r_{1}+r_{2}}\big)+\psi\big(\frac{r_{1}}{r_{1}+r_{2}}\big)},
\end{equation}
where $\psi=\psi_{0}+\gamma$, $\psi_{0}$ is digamma function, the logarithmic derivative of the gamma function, $\gamma$ is Euler-Mascheroni constant (see Remark \ref{rek thm1} for more details about $\psi_{0}$ and $\gamma$),
\begin{align}\label{CH min}
C^{H}_{\min}&=
\sum_{k=0}^{\infty}\frac{1}{k+\frac{r_{\min}}{r_{1}+r_{2}}}H\left(\frac{r_{1}r_{2}}{k(r_{1}+r_{2})+r_{\min}}, 0, 0\right)\nonumber\\
&\quad-\sum_{k=0}^{\infty}\frac{1}{k+1}H\left(-\frac{r_{1}r_{2}}{(k+1)(r_{1}+r_{2})}, 0, 0\right),
\end{align}
and
\begin{align}\label{CH max}
C^{H}_{\max}&=
\sum_{k=0}^{\infty}\frac{1}{k+1}H\left(\frac{r_{1}r_{2}}{(k+1)(r_{1}+r_{2})}, 0, 0\right)\nonumber\\
&\quad-\sum_{k=0}^{\infty}\frac{1}{k+\frac{r_{\max}}{r_{1}+r_{2}}}H\left(-\frac{r_{1}r_{2}}{k(r_{1}+r_{2})+r_{\max}}, 0, 0\right).
\end{align}
Then we have

\begin{theorem}\label{theorem1}
Suppose that $r_{1}, r_{2}\gg\varepsilon$. Then for $x\in\mathbb{R}^{3}\setminus\overline{\mathfrak{B}_{1}\bigcup \mathfrak{B}_{2}}$, if $\rho({\bf{x}})\leq r_{\max}|\log\frac{\varepsilon}{r_{\max}}|^{-2}$, then as $\varepsilon\rightarrow0$, we have
\begin{align}\label{nabla u}
\nabla u(\textbf{x})
=\frac{1}{|\log\varepsilon|}
\frac{\Psi(r_{1}, r_{2})+O\big(\sqrt{\varepsilon}\big|\log\varepsilon\big|\big)}{\varepsilon+\frac{1}{4}\left(\frac{1}{r_{1}}+\frac{1}{r_{2}}\right)\rho({\bf{x}})^{2}}
\left({\bf{n}}
+\Big(\frac{r_{1}+r_{2}}{r_{\min}}\Big)^{2}\eta({\bf{x}})\right)+\nabla g(\bf{x}),
\end{align}
where $\textbf{n}=(1, 0, 0)$, the blowup factor $\Psi(r_{1}, r_{2})$ is given by \eqref{def psi}, $|\nabla g|$ is bounded on any bounded region in $\mathbb{R}^{3}\setminus \overline{\mathfrak{B}_{1}\cup \mathfrak{B}_{2}}$ regardless of $r_{1}, r_{2}$ and $\varepsilon$ ($g$ is defined by \eqref{g}),  and
$$|\eta(\textbf{x})|\leq C|\log\varepsilon|^{-1}$$
for some positive constant $C>0$ independent of $\varepsilon$, $r_{1}$, and $r_{2}$.
\end{theorem}

\begin{remark}\label{rek thm1}
The digamma function
$$\psi_{0}(x)=\int^{\infty}_{0}\bigg(\frac{e^{-t}}{t}-\frac{e^{-xt}}{1-e^{-t}}\bigg)\ dt.$$
Especially, if $r_{1}=r_{2}=r$, then
$$\psi(\frac{r_{1}}{r_{1}+r_{2}})=\psi(\frac{r_{2}}{r_{1}+r_{2}})=\psi(\frac{1}{2})=\psi_{0}\big(\frac{1}{2}\big)+\gamma=-2\log2,$$ where $\gamma=\lim\limits_{m\rightarrow\infty}\left(\sum_{k=1}^{m}\frac{1}{k}-\log m\right)$ is the Euler-Mascheroni constant, which
is an irrational number, $0.577215\cdots$. At this moment,
\begin{align*}
C_{\min}^{H}&=\sum_{k=0}^{\infty}\frac{1}{k+\frac{1}{2}}H\left(\frac{r}{2k+1},0, 0\right)-\sum_{k=0}^{\infty}\frac{1}{k+1}H\left(-\frac{r}{2(k+1)}, 0,0\right),\\
C^{H}_{\max}&=
\sum_{k=0}^{\infty}\frac{1}{k+1}H\left(\frac{r}{2(k+1)}, 0, 0\right)-\sum_{k=0}^{\infty}\frac{1}{k+\frac{1}{2}}H\left(-\frac{r}{2k+1}, 0, 0\right).
\end{align*}
Thus,
\begin{align*}
\Psi(1, 1)&=\frac{1}{2}\big(C^{H}_{\min}+C^{H}_{\max}\big)\\
&=\sum_{k=0}^{\infty}\frac{1}{2k+2}\left(H\Big(\frac{r}{2k+2}, 0,0\Big)-H\Big(-\frac{r}{2k+2}, 0, 0\Big)\right)\\
&\quad+\sum_{k=0}^{\infty}\frac{1}{2k+1}\left(H\left(\frac{r}{2k+1}, 0, 0\right)-H\Big(-\frac{r}{2k+1},0, 0\Big)\right).
\end{align*}
It is exactly $\frac{1}{2}C_{H}$ defined by (1.17) in \cite{KLY2}. So that, we have the main conclusion of \cite{KLY2},
\begin{equation*}
\nabla u(\textbf{x})=\Big(C_{H}+O\big(\sqrt{\varepsilon}\big|\log\varepsilon\big|\big)\Big)\frac{{\bf{n}}+\eta(\textbf{x})}{|\log \varepsilon|(2\varepsilon+\frac{1}{r}\rho(\textbf{x})^{2})}+\nabla g(\textbf{x}).
\end{equation*}
There is a typo in (1.12) in \cite{KLY2} that $r\rho(\textbf{x})^{2}$ in the denominator should be $\frac{1}{r}\rho(\textbf{x})^{2}$.

We would like to thank Mikyoung Lim for informing us the work \cite{LY} after we finished our draft. In \cite{LY}, they mainly use the bispherical coordinate system and the Euler-Maclaurin formula motivated by the physical intuition to obtain the quantity $\Psi(r_{1},r_{2})$ as well. However, this paper is mathematically along the line of \cite{KLY2,LY1} and completely improves the results in \cite{KLY2} to more general case $r_{1}\neq r_{2}$.
\end{remark}

From \eqref{nabla u}, we can infer that a high concentration of extreme electric filed occurs when $\rho({\bf{x}})=0$; that is, when ${\bf{x}}$ is on the line segment connecting two closest points on the two spheres, we have
$$\nabla u\simeq\frac{\Psi(r_{1}, r_{2})}{\varepsilon|\log\varepsilon|}{\bf{n}}.$$
From this, the occurrence of the gradient blowup depends on the behavior of $\Psi(r_{1}, r_{2})$.  The explicit formula of the blowup factor $\Psi(r_{1}, r_{2})$ expressed by \eqref{def psi} is the main contribution of this paper. To identify its role, let us see the following examples.

\begin{figure}[ht]
\begin{minipage}[t]{0.45\textwidth}
  {
  \includegraphics[width=6.5cm]{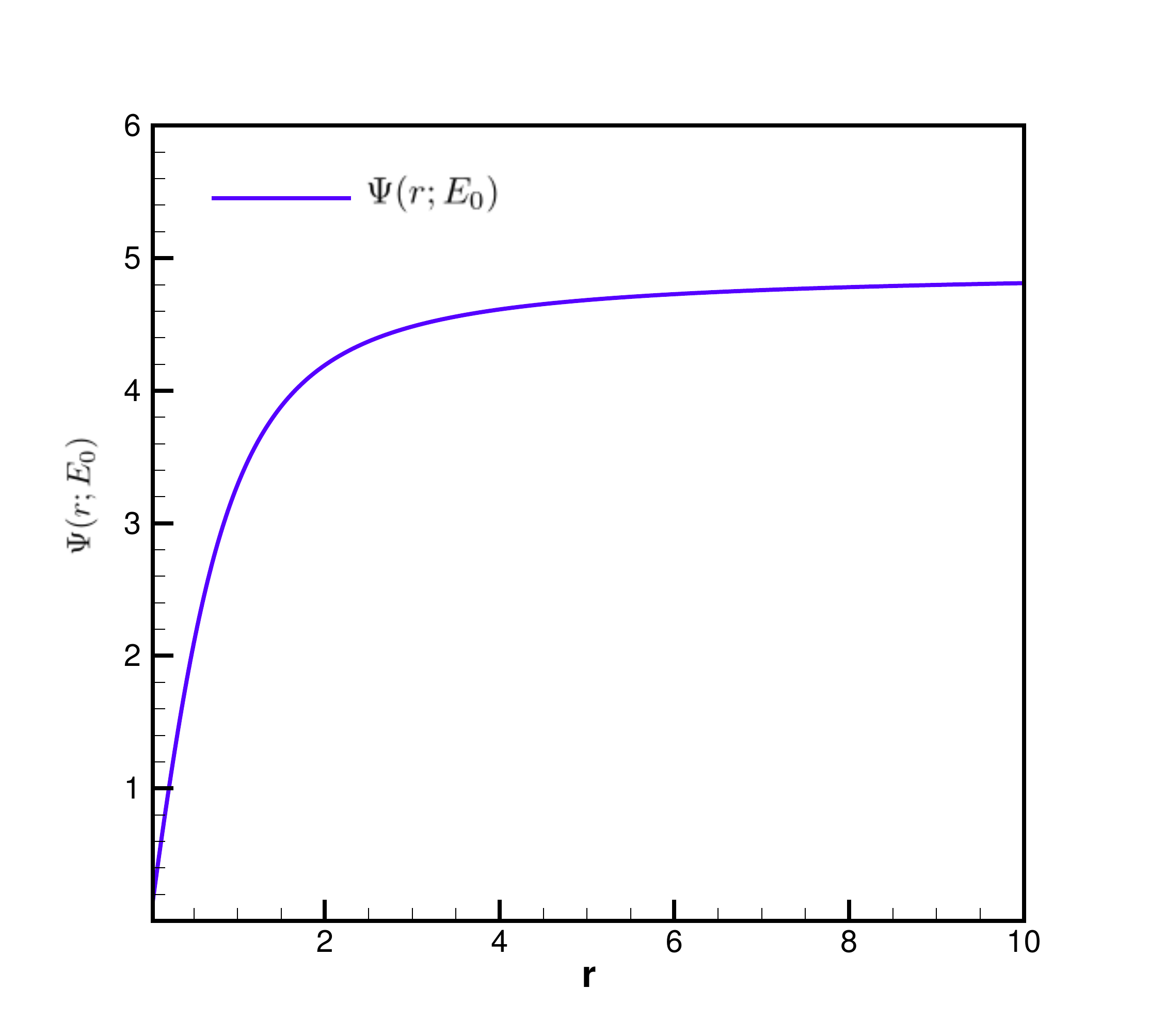}
 \caption{Graph of $\Psi(r;E_{0})$.}}\label{fig1}
\end{minipage}
\quad
\begin{minipage}[t]{0.45\textwidth}
  {
  \includegraphics[width=6.5cm]{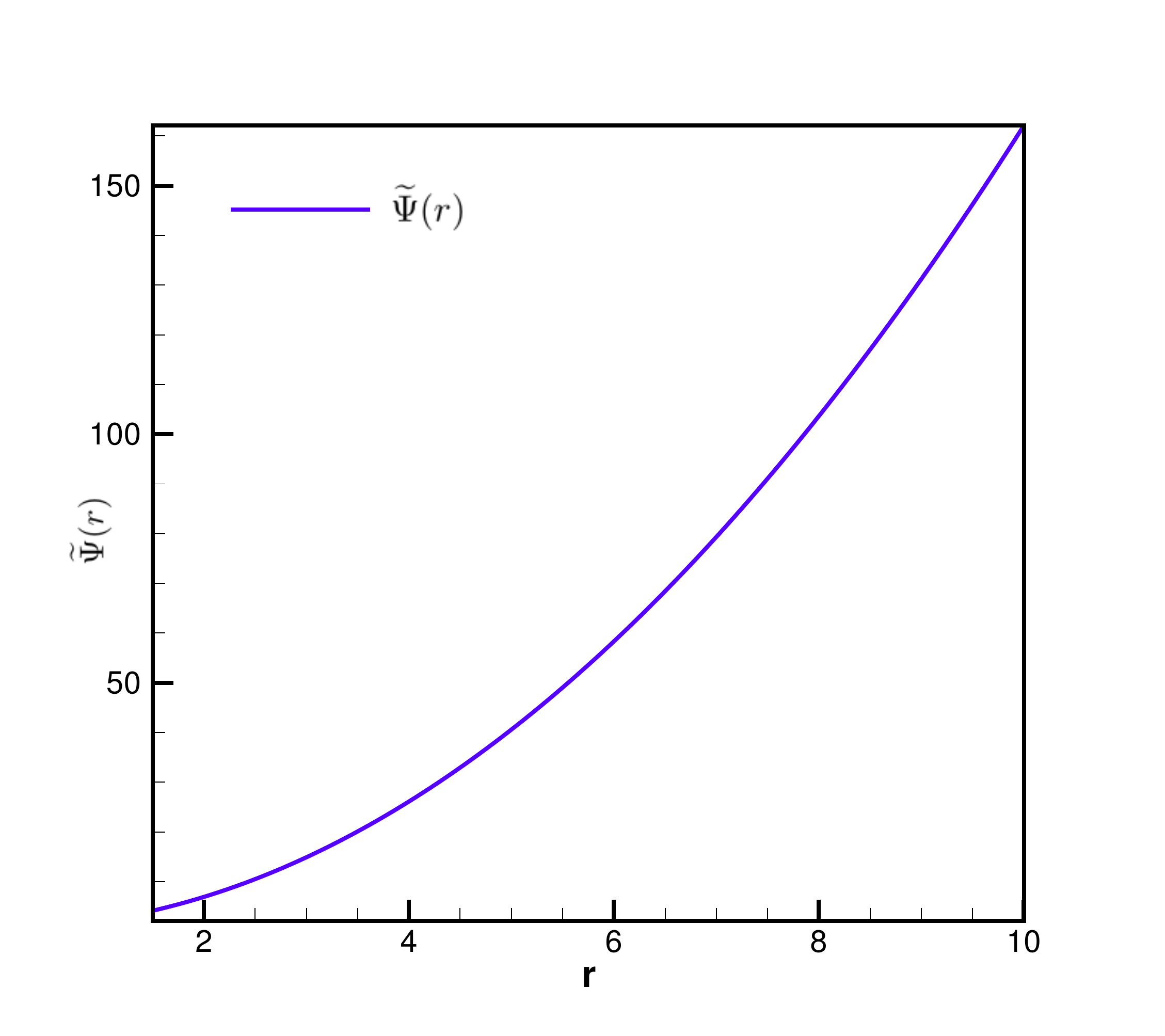}
 \caption{Graph of $\widetilde\Psi(r)$.}}\label{fig2}
\end{minipage}
\end{figure}
\begin{example}\label{example}
First, if $H=E_{0}x_{2}$ for $E_{0}>0$, then it is easy to see that $C_{\min}^{H}=C_{\max}^{H}=0$, so $\Psi(r_{1},r_{2})=0$. This means that there is no blowup occurring. Now we assume that $H=E_{0}x_{1}$, $E_{0}>0$. Without loss of generality, we assume $r_{1}=1$, and $r_{2}=r$. Then we have the following two cases:
(1)
If $r\leq1$, then $r_{\max}=1$ and $r_{\min}=r$. Thus
\begin{align*}
C^{H}_{\min}&=\sum_{k=0}^{\infty}\frac{E_{0}r}{(k+1)^{2}(r+1)}+\sum_{k=0}^{\infty}\frac{E_{0}r(1+r)}{\Big(k(r+1)+r\Big)^{2}}\nonumber\\
&=\frac{E_{0}r}{1+r}\left(\sum_{k=1}^{\infty}\frac{1}{k^{2}}+\sum_{k=1}^{\infty}\frac{1}{\big(k-\frac{1}{1+r})^{2}}\right)\nonumber\\
&=\frac{E_{0}r}{1+r}\left(\frac{\pi^{2}}{6}+\psi'\Big(\frac{r}{1+r}\Big)\right),
\end{align*}
and
\begin{align*}
C^{H}_{\max}=\frac{E_{0}r}{1+r}\left(\frac{\pi^{2}}{6}+\psi'\Big(\frac{1}{1+r}\Big)\right),
\end{align*}
where $\psi'$ is the first derivative of $\psi$.
(2)
If $r>1$, then
\begin{align*}
C^{H}_{\min}
=\frac{E_{0}r}{1+r}\left(\frac{\pi^{2}}{6}+\psi'\Big(\frac{1}{1+r}\Big)\right),\quad
C^{H}_{\max}=\frac{E_{0}r}{1+r}\left(\frac{\pi^{2}}{6}+\psi'\Big(\frac{r}{1+r}\Big)\right).
\end{align*}
It is easy to see that $C^{H}_{\min}, C^{H}_{\max}\neq0$.

Setting $\Psi(r;E_{0}):=\Psi(1,r)$, we have
\begin{align*}
\Psi(r;E_{0})&=\frac{E_{0}r}{1+r}\left(\frac{\pi^{2}}{6}+\frac{\psi\big(\frac{1}{1+r}\big)\psi'\big(\frac{r}{1+r}\big)+\psi\big(\frac{r}{1+r}\big)\psi'\big(\frac{1}{1+r}\big)}
{\psi\big(\frac{1}{1+r}\big)+\psi\big(\frac{r}{1+r}\big)}\right),
\end{align*}
which is strictly positive. This implies $|\nabla u|$ blows up for sufficiently small $\varepsilon$. For fixed $E_{0}$, one can see from Figure \ref{fig1} that $\Psi(r;E_{0})$ is increasing with respect to $r$. On the other hand, $\Psi(r;E_{0})$ is also increasing with respect to $E_{0}\in(0,\infty)$, for fixed $r$.
\end{example}

\begin{example}
Assume that $H=x_{1}^{3}-3x_{1}x_{2}^{2}$, $r_{1}=1$, and $r_{2}=r$. In this case, we denote $\widetilde\Psi(r):=\Psi(1,r)$ and get
\begin{align*}
\widetilde\Psi(r)&=\left(\frac{r}{1+r}\right)^{3}\left(\frac{\pi^{4}}{90}+\frac{1}{6}\frac{\psi\big(\frac{1}{1+r}\big)\psi^{(3)}\big(\frac{r}{1+r}\big)+\psi\big(\frac{r}{1+r}\big)\psi^{(3)}\big(\frac{1}{1+r}\big)}
{\psi\big(\frac{1}{1+r}\big)+\psi\big(\frac{r}{1+r}\big)}\right)>0,
\end{align*}
where $\psi^{(3)}$ is the third derivative of $\psi$. Therefore, $|\nabla u|$ blows up as $\varepsilon\rightarrow0$. Moreover, $\widetilde\Psi(r)$ is increasing with respect to $r$ (see Figure 2).
\end{example}

The remainder of this paper is organized as follows. In section \ref{pf thm}, we give the outline of the proof of Theorem \ref{theorem1} and reduce the proof of Theorem \ref{theorem1} to establishing the asymptotic formulae of $Q_{1}, Q_{2}, M$, the singular function $h(\textbf{x})$, and $u|_{\partial \mathfrak{B}_{1}}-u|_{\partial \mathfrak{B}_{2}}$. In Section \ref{sec pf prop Q}, we deal with the asymptotic formulae of $Q_{1}, Q_{2}$, and $M$ by exploring several properties of the sequences ${\textbf{p}_{i,j}}$ and ${q_{i,j}}$. In section \ref{sec pf prop h} we are devoted to the proof of Proposition \ref{prop h}, which characterizes the asymptotic behavior of $\nabla h(\textbf{x})$. The asymptotic formula of $u|_{\partial \mathfrak{B}_{1}}-u|_{\partial \mathfrak{B}_{2}}$ is given in Section \ref{sec pf prop u}.

\section{Proof of Theorem \ref{theorem1}}\label{pf thm}

In this section, we are devoted to proving Theorem \ref{theorem1}. We first introduce a singular function $h(\textbf{x})$ and establish its asymptotic formula by making use of our improvement on $Q_{1}, Q_{2}$, and $M$. Then we further investigate the asymptotic formula of $\nabla h(\textbf{x})$, and obtain the blowup factor $\Psi(r_{1},r_{2})$. We follow the notations in \cite{LY1}.

The main ingredient to prove Theorem \ref{theorem1} is the singular function $h$, first introduced in \cite{Y1}, which is the solution to
\begin{equation}\label{h}
\begin{cases}
\Delta h=0, &\mbox{in}~ \mathbb{R}^{3}\setminus \overline{\mathfrak{B}_{1}\cup \mathfrak{B}_{2}},\\
h=\mbox{constant}, &\mbox{on}~\partial \mathfrak{B}_{i}, i=1, 2,\\
\int _{\partial D_{i}}\frac{\partial h}{\partial\nu^{i}}\ d\sigma=(-1)^{i+1}, &i=1, 2,\\
h=O(|\textbf{x}|^{-2}), &\mbox{as}~|\textbf{x}|\rightarrow\infty.
\end{cases}
\end{equation}
The existence and uniqueness of the solution can be referred to \cite{ACKLY,Y1}. We emphasize that the constant values of $h$ on $\partial \mathfrak{B}_{1}$ and on $\partial \mathfrak{B}_{2}$ are different. So that $\nabla h$ becomes arbitrary large if $\varepsilon$ goes to zero. Define the function $g$ by
\begin{equation}\label{g}
u(\textbf{x})=\frac{u|_{\partial \mathfrak{B}_{2}}-u|_{\partial \mathfrak{B}_{1}}}{h|_{\partial \mathfrak{B}_{2}}-h|_{\partial \mathfrak{B}_{1}}}h(\textbf{x})+g(\textbf{x}), \quad \textbf{x}\in \mathbb{R}^{3}\setminus (\mathfrak{B}_{1}\cup \mathfrak{B}_{2}).
\end{equation}
Then one can see that $g$ is harmonic in $\mathbb{R}^{3}\setminus \overline{\mathfrak{B}_{1}\cup \mathfrak{B}_{2}}$ and $g|_{\partial \mathfrak{B}_{1}}=g|_{\partial \mathfrak{B}_{2}}$; that is, there is no potential difference of $g$ on $\partial \mathfrak{B}_{1}$ and $\partial \mathfrak{B}_{2}$. By using the same way as in \cite{KLY1,LLBY}, we can show that $|\nabla g|$ is bounded on any bounded subset of $\mathbb{R}^{3}\setminus (\mathfrak{B}_{1}\cup \mathfrak{B}_{2})$. Thus, the function $h$ characterizes the singular behavior of the solution to \eqref{instruction}, and the singular behavior of $\nabla u$ is determined by
$$\frac{u|_{\partial \mathfrak{B}_{2}}-u|_{\partial \mathfrak{B}_{1}}}{h|_{\partial \mathfrak{B}_{2}}-h|_{\partial \mathfrak{B}_{1}}}\nabla h(\textbf{x}).$$
Therefore, the proof of Theorem \ref{theorem1} is reduced to the estimates or expansions of $u|_{\partial \mathfrak{B}_{2}}-u|_{\partial \mathfrak{B}_{1}}, h|_{\partial \mathfrak{B}_{2}}-h|_{\partial \mathfrak{B}_{1}}$, and $\nabla h(\textbf{x})$.

To this end, we introduce the following notations. Let $R_{i}$ be the reflection with respect to $\partial \mathfrak{B}_{i}$, $i=1, 2$, i.e.,
$$R_{i}(\textbf{x})=\frac{r_{i}^{2}(\textbf{x}-\textbf{c}_{i})}{|\textbf{x}-\textbf{c}_{i}|^{2}}+\textbf{c}_{i}.$$
Denote
\begin{align}\label{p_{1,2k}p_{1,2k+1}}
\begin{cases}
\textbf{p}_{1, 2k}:=(R_{1}R_{2})^{k}\textbf{c}_{1},&\quad\mbox{in}~\mathfrak{B}_{1},\\
\textbf{p}_{1, 2k+1}:=R_{2}(R_{1}R_{2})^{k}\textbf{c}_{1},&\quad\mbox{in}~\mathfrak{B}_{2},
\end{cases}
\quad k=0, 1, \ldots,
\end{align}
and
\begin{equation}\label{q_{1}}
q_{1, 0}=1, \quad q_{1, j}=\prod_{l=0}^{j}\mu_{1, l}, \quad j=1, 2, \ldots,
\end{equation}
where
\begin{equation}\label{mu_{1}}
\mu_{1, j}=
\begin{cases}
1, &\mbox{if}~j=0,\\
\frac{r_{1}}{|\textbf{c}_{1}-\textbf{p}_{1, 2k-1}|}, &\mbox{if}~j=2k,~k\geq1,\\
\frac{r_{2}}{|\textbf{c}_{2}-\textbf{p}_{1, 2k}|}, &\mbox{if}~j=2k+1,~k\geq0.
\end{cases}
\end{equation}

Similarly,
\begin{equation*}\label{p_{2}}
\begin{cases}
\textbf{p}_{2, 2k}:=(R_{2}R_{1})^{k}\textbf{c}_{2},&\quad\mbox{in}~\mathfrak{B}_{2},\\
\textbf{p}_{2, 2k+1}:=R_{1}(R_{2}R_{1})^{k}\textbf{c}_{2},&\quad\mbox{in}~\mathfrak{B}_{1},\\
\end{cases}
\quad k=0, 1, \ldots,
\end{equation*}
and
\begin{equation*}\label{q_{2}}
q_{2, 0}=1, \quad q_{2, j}=\prod_{l=0}^{j}\mu_{2, l}, \quad j=1, 2, \ldots,
\end{equation*}
where
\begin{equation}\label{mu_{2}}
\mu_{2, j}=
\begin{cases}
1, &\mbox{if}~j=0,\\
\frac{r_{2}}{|\textbf{c}_{2}-\textbf{p}_{2, 2k-1}|}, &\mbox{if}~j=2k,~k\geq1, \\
\frac{r_{1}}{|\textbf{c}_{1}-\textbf{p}_{2, 2k}|}, &\mbox{if}~j=2k+1,~k\geq0.
\end{cases}
\end{equation}

Set
\begin{equation}\label{def Qs}
Q_{i}=\sum\limits_{j=0}^{\infty}(-1)^{j}q_{i,j},\quad i=1, 2,
\end{equation}
and
\begin{equation}\label{def M}
M=Q_{2}\sum\limits_{k=0}^{\infty}q_{1, 2k}+Q_{1}\sum\limits_{k=0}^{\infty}q_{2, 2k+1}=Q_{1}\sum\limits_{k=0}^{\infty}q_{2, 2k}+Q_{2}\sum\limits_{k=0}^{\infty}q_{1, 2k+1}.
\end{equation}

By using image charge method, Lim and Yun \cite{LY1} obtained the following expression of $h(\textbf{x})$, which has been used to derive estimates like \eqref{lower-upper}.
\begin{lemma}\label{lem2.1}
The solution to \eqref{h} is given by
\begin{equation}\label{def h}
h(\textbf{x})=-\frac{Q_{2}}{M}\sum_{j=0}^{\infty}(-1)^{j}q_{1,j}\Gamma(\textbf{x}-\textbf{p}_{1, j})+\frac{Q_{1}}{M}\sum_{j=0}^{\infty}(-1)^{j}q_{2,j}\Gamma(\textbf{x}-\textbf{p}_{2,j}),
\end{equation}
where $\Gamma(\textbf{x})=\frac{1}{4\pi}|\textbf{x}|^{-1}$ is the fundamental solution of the Laplacian in three dimensions.
\end{lemma}

The following Proposition \ref{prop Q} gives the complete expressions of $Q_{1}$, $Q_{2}$, and the asymptotic formula of $M$.

\begin{prop}\label{prop Q}
For $Q_ {1}$, $Q_ {2}$, and $M$ defined by \eqref{def Qs} and \eqref{def M}, we have
$$Q_{1}=-\frac{r_{2}}{r_{1}+r_{2}}\psi\Big(\frac{r_{2}}{r_{1}+r_{2}}\Big),
\quad Q_{2}=-\frac{r_{1}}{r_{1}+r_{2}}\psi\Big(\frac{r_{1}}{r_{1}+r_{2}}\Big),$$
and
\begin{align*}
M=-\frac{1}{2}\frac{r_{1}r_{2}}{(r_{1}+r_{2})^{2}}|\log\varepsilon|
\left(\psi\Big(\frac{r_{1}}{r_{1}+r_{2}}\Big)+\psi\Big(\frac{r_{2}}{r_{1}+r_{2}}\Big)\right)
\left(1+\frac{r_{1}+r_{2}}{r_{2}}O\left(|\log\varepsilon|^{-1}\right)\right),
\end{align*}
where $\psi_{0}$ is digamma function, $\gamma$ is Euler-Mascheroni constant, $O\left(|\log\varepsilon|^{-1}\right)$ is independent of $r$.
\end{prop}

We remark that Lim and Yun in \cite{LY1} obtained the upper and lower bounds of $|\nabla u|$ by using the estimates
\begin{equation}\label{est Qi}
\frac{1}{C}\big(\frac{r_{i}}{r_{1}+r_{2}}\big)\leq Q_{i}\leq C\big(\frac{r_{i}}{r_{1}+r_{2}}\big),\quad i=1,2.
\end{equation}
Proposition \ref{prop Q} is an important improvement on $Q_{i}$, $i=1,2$, which is the first difficulty that we overcome in this paper.

Let
\begin{align*}
R_{\delta}: =\left\{\textbf{x}\in \mathbb{R}^{3}\setminus (\mathfrak{B}_{1}\cup \mathfrak{B}_{2})\big|~~\rho(\textbf{x})\leq r_{1}|\log\delta|^{-2}\right\}
\end{align*}
be a narrow region in between $\mathfrak{B}_{1}$ and $\mathfrak{B}_{2}$, where $\delta=\frac{\varepsilon}{r_{1}}$. Then we have the asymptotic formula of $\nabla h(\textbf{x})$ in $R_{\delta}$.

\begin{prop}\label{prop h}
For ${\bf{x}}\in R_{\delta}$, we have
\begin{equation*}
\nabla h(\textbf{x})=-\frac{r_{1}+r_{2}}{4\pi r_{1}r_{2}|\log\varepsilon|}\frac{1}{\varepsilon+\frac{1}{4}\left(\frac{1}{r_{1}}+\frac{1}{r_{2}}\right)\rho({\bf{x}})^{2}}\left({\bf{n}}
+\Big(\frac{r_{1}+r_{2}}{r_{2}}\Big)^{2}O(|\log\varepsilon|^{-1})\right).
\end{equation*}
\end{prop}

From \eqref{p_{1,2k}p_{1,2k+1}}--\eqref{mu_{2}} and Lemma \ref{lem2.1}, it is not difficult to see that
\begin{align}\label{h-difference}
&\frac{h(\textbf{x})}{h|_{\partial \mathfrak{B}_{1}}-h|_{\partial \mathfrak{B}_{2}}}=(-4\pi M)\left(\frac{Q_{2}}{r_{1}}+\frac{Q_{1}}{r_{2}}\right)^{-1}h(\textbf{x})\nonumber\\
=&\frac{r_{1}r_{2}}{r_{2}Q_{2}+r_{1}Q_{1}}\Bigg(Q_{2}\sum_{j=0}^{\infty}\frac{(-1)^{j}q_{1,j}}{|\textbf{x}-\textbf{p}_{1, j}|}-Q_{1}\sum_{j=0}^{\infty}\frac{(-1)^{j}q_{2,j}}{|\textbf{x}-\textbf{p}_{2,j}|}\Bigg).
\end{align}
Finally, substituting these estimates above into the relationship \cite{Y1,Y2}
\begin{equation}\label{u_{1}-u_{2}}
u|_{\partial \mathfrak{B}_{1}}-u|_{\partial \mathfrak{B}_{2}}=\int_{\partial \mathfrak{B}_{1}}H\frac{\partial h}{\partial \nu^{1}}\ d\sigma+\int_{\partial \mathfrak{B}_{2}}H\frac{\partial h}{\partial \nu^{2}}\ d\sigma,
\end{equation}
we have

\begin{prop}\label{prop u}
As $\varepsilon\rightarrow0$, we have
\begin{equation*}
u|_{\partial \mathfrak{B}_{1}}-u|_{\partial \mathfrak{B}_{2}}=\frac{2\Psi(r_{1}, r_{2})+O\big(\sqrt{\varepsilon}\big|\log\varepsilon\big|\big)}{|\log\varepsilon|}\left(1+\frac{r_{1}+r_{2}}{r_{2}}O(|\log\varepsilon|^{-1})\right),
\end{equation*}
where $\Psi(r_{1}, r_{2})$ is defined by \eqref{def psi}.
\end{prop}

Now, we are ready to prove Theorem \ref{theorem1}.
\begin{proof}[\bf Proof of Theorem \ref{theorem1}.]
For $\textbf{x}\in R_{\delta}$, by using \eqref{h-difference} and Proposition \ref{prop Q}, we have
\begin{equation*}
\frac{\nabla h(\textbf{x})}{h|_{\partial \mathfrak{B}_{1}}-h|_{\partial \mathfrak{B}_{2}}}=\frac{1}{2}\frac{1}{\varepsilon+\frac{1}{4}\left(\frac{1}{r_{1}}+\frac{1}{r_{2}}\right)\rho({\bf{x}})^{2}}\left({\bf{n}}
+\Big(\frac{r_{1}+r_{2}}{r_{2}}\Big)^{2}O(|\log\varepsilon|^{-1})\right).
\end{equation*}
Thus, combining \eqref{g} and Proposition \ref{prop u}, we have
\begin{align*}
\nabla u(\textbf{x})&=\frac{u|_{\partial \mathfrak{B}_{2}}-u|_{\partial \mathfrak{B}_{1}}}{h|_{\partial \mathfrak{B}_{2}}-h|_{\partial \mathfrak{B}_{1}}}\nabla h(\textbf{x})+\nabla g(\textbf{x})\\
&=\frac{1}{|\log\varepsilon|}
\frac{\Psi(r_{1}, r_{2})+O\big(\sqrt{\varepsilon}\big|\log\varepsilon\big|\big)}{\varepsilon+\frac{1}{4}\left(\frac{1}{r_{1}}+\frac{1}{r_{2}}\right)\rho({\bf{x}})^{2}}
\left({\bf{n}}
+\Big(\frac{r_{1}+r_{2}}{r_{2}}\Big)^{2}O(|\log\varepsilon|^{-1})\right)+\nabla g(\bf{x}).
\end{align*}
Thus, Theorem \ref{theorem1} is proved.
\end{proof}

\begin{remark}
We now compare with the result in \cite{KLY2} for $r_{1}=r_{2}$. When $r_{1}=r_{2}$, the computation becomes easy to handle. In fact, by using the symmetry and \eqref{p_{1,2k}p_{1,2k+1}}--\eqref{mu_{2}}, we have
\begin{align*}
\textbf{p}_{1, 2k}=-\textbf{p}_{2, 2k}, \quad \textbf{p}_{1, 2k+1}=-\textbf{p}_{2, 2k+1},\quad q_{1, 2k}=q_{2, 2k}, \quad q_{1, 2k+1}=q_{2, 2k+1}.
\end{align*}
In this case, we can rewrite
$$\textbf{p}_{2k}:=\textbf{p}_{2,2k}=-\textbf{p}_{1,2k},\quad\textbf{p}_{2k+1}:=\textbf{p}_{1,2k+1}=-\textbf{p}_{2,2k+1},\quad q_{j}:=q_{1,j}=q_{2,j},\quad j\in\mathbb N.$$
Hence,
\begin{equation*}
Q_{1}=Q_{2}=-\frac{1}{2}\psi\Big(\frac{1}{2}\Big),\quad M=-\frac{1}{2}\psi\Big(\frac{1}{2}\Big)\sum_{j=0}^{\infty}q_{j},
\end{equation*}
and \eqref{def h} becomes
$$h(x)=\frac{1}{\sum_{j=0}^{\infty}q_{j}}\sum\limits_{j=0}^{\infty}q_{j}\left(\Gamma(\textbf{x}-\textbf{p}_{j})-\Gamma(\textbf{x}+\textbf{p}_{j})\right),$$
which is the same as (1.22) in \cite{KLY2}. For general case $r_{1}\neq r_{2}$, we should find the explicit expression of $Q_{1}$, $Q_{2}$ and $M$ in terms of $r_{1}$ and $r_{2}$.
\end{remark}

\section{Proof of Proposition \ref{prop Q}}\label{sec pf prop Q}

To prove Proposition \ref{prop Q}, from the definitions of $Q_{1}, Q_{2}$ and $M$, \eqref{def Qs} and  \eqref{def M}, we need to study some properties of the sequences ${\textbf{p}_{i,j}}$ and ${q_{i,j}}$ for $i=1, 2$, $j\in \mathbb{N}$. Differently from the special case when $r_{1}=r_{2}$ in \cite{KLY2}, where the symmetry of the domain makes the computation much easier to deal with, we now have to find the leading terms of ${\textbf{p}_{i,j}}$, ${q_{i,j}}$ in terms of $r_{1}$ and $r_{2}$.

\subsection{Properties of the sequences ${\textbf{p}_{i,j}}$ and ${q_{i,j}}$}

In the following, we assume without loss of generality that $r_{1}>r_{2}$. Set
$$r=\frac{r_{2}}{r_{1}}\quad\quad\mbox{and}\quad~\delta=\frac{\varepsilon}{r_{1}}.$$
We only consider the case when $r\leq1$.
If $r>1$, then we replace $r$ and $\delta$ by $\frac{1}{r}$ and $\frac{\varepsilon}{r_{2}}$, respectively. We fix our notations now. For $\textbf{p}_{i,j}=(p_{i,j}, 0, 0)$, we denote
$${\pmb{\mathscr{P}}}_{i,j}:=\frac{\textbf{p}_{i,j}}{r_{1}}=:(\mathscr{P}_{i,j}, 0, 0), \quad i=1, 2, ~j\in \mathbb{N}.$$
Let $\textbf{p}_{1}=(p_{1}, 0, 0)\in \mathfrak{B}_{1}$ be the fixed point of $R_{1}R_{2}$, then
$\textbf{p}_{2}=R_{2}(\textbf{p}_{1})=(\textbf{p}_{2}, 0, 0)\in \mathfrak{B}_{2}$ is the fixed point of $R_{2}R_{1}$. We emphasize that $\textbf{p}_{i,j}\in \mathfrak{B}_{1}$ decreases to $\textbf{p}_{1}$ if $i+j$ is odd, and $\textbf{p}_{i,j}\in \mathfrak{B}_{2}$ increases to $\textbf{p}_{2}$ if $i+j$ is even, $i=1,2, j\in\mathbb N$. For readers' convenience, we now list some results obtained in subsection 4.3 of \cite{LY1} as follows.
\begin{equation}\label{p_{1}}
\mathscr{P}_{i}:=\frac{p_{i}}{r_{1}}=(-1)^{i-1}2\Big(\frac{r\delta}{r+1}\Big)^{1/2}+O_{i}(\delta),\quad i=1,2,
\end{equation}
and the sequence $\mathscr{P}_{1, j}$ can be expressed as
\begin{equation}\label{2.4}
\mathscr{P}_{1,2k}=\left(\left(\frac{1}{1+\delta-\mathscr{P}_{1}}+B_{1}\right)A_{1}^{k}-B_{1}\right)^{-1}+\mathscr{P}_{1},
\end{equation}
\begin{equation}\label{2.5}
\mathscr{P}_{1,2k+1}=\left(\left(\frac{1}{-\frac{r}{r+1}+O(\delta)-\mathscr{P}_{2}}+B_{2}\right)A_{2}^{k}-B_{2}\right)^{-1}+\mathscr{P}_{2},
\end{equation}
where
\begin{equation}\label{A_{1}}
A_{i}
=1+4\Big(\frac{r+1}{r}\delta\Big)^{1/2}+\frac{r+1}{r}O_{i}(\delta),\quad i=1,2,
\end{equation}
and
\begin{equation}\label{B_{1}}
\sqrt{\delta}B_{i}=(-1)^{i-1}\frac{1}{4}\Big(\frac{r+1}{r}\Big)^{1/2}+\frac{r+1}{r}O_{i}(\sqrt{\delta}),\quad i=1,2.
\end{equation}

For the sake of convenience, we will only deal with the case of $i=1$ for instance, since the argument for $i=2$ is the same. Recalling that $\mathscr{P}_{1,2k}\in \mathfrak{B}_{1}$ and $\mathscr{P}_{1,2k+1}\in \mathfrak{B}_{2}$, then for simplicity, we use
\begin{equation}\label{def Theta}
\Theta_{1,2k}(r):=\frac{r}{k(r+1)+r},\quad\Theta_{1,2k+1}(r):=\frac{r}{(k+1)(r+1)}
\end{equation}
to denote the main terms of $\mathscr{P}_{i,j}$ and $q_{i,j}$, $i=1,2$, $j\in\mathbb{N}$. We first choose an approximate number
\begin{equation*}
N(\delta)=\min_{k\in \mathbb{N}}\left\{k\geq \frac{1}{\widetilde{C}}\frac{\log2}{8}\Big(\frac{r}{r+1}\Big)^{1/2}\frac{1}{\sqrt{\delta}}\right\}
\end{equation*}
which is fixed in \cite{LY1} and $\tilde{C}>0$ is a constant independent of $\varepsilon,r$, and $k$, so that the sequence terms of $k\leq N(\delta)$ are dominant in the sequences $\mathscr{P}_{1,2k}$ and $\mathscr{P}_{1,2k+1}$.

\begin{lemma}\label{lem3.1}
Let $N(\delta)>0$ be defined as above. If $k\leq N(\delta)$, we have
\begin{equation}\label{y_{1,2k}}
\big|\mathscr{P}_{1,2k}-\Theta_{1,2k}(r)\big|,~\big|\mathscr{P}_{1,2k+1}+\Theta_{1,2k+1}(r)\big|\leq C\Big(\frac{r\delta}{r+1}\Big)^{1/2},
\end{equation}
and
\begin{equation}\label{q_{1,2k}}
\big|q_{1,2k}-\Theta_{1,2k}(r)\big|,~\big|q_{1,2k+1}-\Theta_{1,2k+1}(r)\big|\leq C\Big(\frac{r+1}{r}\delta\Big)^{1/2},
\end{equation}
where $C>0$ is independent of $r$ and $\delta$.
\end{lemma}

The proof is very similar with that of Lemma 4.2 in \cite{LY1}. We omit it here.

For a given $\varepsilon>0$, let $N_{0}(\delta)$ and $N_{1}(\delta)$ be as follows:
\begin{equation*}
N_{0}(\delta)=[|\mbox{log}\delta|]
\end{equation*}
and
\begin{equation*}
N_{1}(\delta)=\left[\frac{1}{\delta|\mbox{log}\delta|}\right].
\end{equation*}
Here $[\cdot]$ is the Gaussian bracket. Since $\delta$ is sufficiently small, we have
$$N_{0}(\delta)\ll N(\delta)\ll N_{1}(\delta).$$

We have the following lemma.
\begin{lemma}\label{lem3.2}
\begin{enumerate}[(i)]
\item\label{lem3.2-1}
There exists a positive constant C independent of $r$ and $\delta$ such that
\begin{equation}\label{2.22}
\Big|\sum_{k=0}^{\infty}q_{1, 2k}-\sum_{k=0}^{N(\delta)-1}\Theta_{1,2k}(r)\Big|,\quad\sum_{k=N(\delta)}^{\infty}q_{1, 2k}\leq C,
\end{equation}
and
\begin{equation}\label{2.22'}
\Big|\sum_{k=0}^{\infty}q_{1, 2k+1}-\sum_{k=0}^{N(\delta)-1}\Theta_{1,2k+1}(r)\Big|, \quad\sum_{k=N(\delta)}^{\infty}q_{1, 2k+1}\leq C.
\end{equation}
\item\label{lem3.2-2}
For all $k\geq 0$, we have
\begin{equation}\label{2.26}
\mathscr{P}_{1, 2k}-\mathscr{P}_{1}\geq2\Big(\frac{r\delta}{r+1}\Big)^{1/2}A^{-k}_{1},
\end{equation}
and
$$\mathscr{P}_{1, 2k+1}-\mathscr{P}_{2}\leq-2\Big(\frac{r\delta}{r+1}\Big)^{1/2}A^{-k}_{2},$$
where $A_{i}$ are defined by \eqref{A_{1}}, $i=1,2$.
\item\label{lem3.2-3}
There exists a positive constant C independent of $r$ and $\delta$ such that for all $k\leq N(\delta)$, we have
\begin{equation}\label{2.30}
\mathscr{P}_{1, 2k}-\mathscr{P}_{1}\geq C\frac{r}{r+1}\frac{1}{k},\quad
\mathscr{P}_{1, 2k+1}-\mathscr{P}_{2}\leq-C\frac{r}{r+1}\frac{1}{k}.
\end{equation}
\item\label{lem3.2-4}
\begin{equation*}
0<\mathscr{P}_{1, 2N_{1}(\delta)}-\mathscr{P}_{1}\leq\exp\left({-\big(\frac{r+1}{r}\big)^{1/2}\frac{1}{\sqrt{\delta}|\log \delta|}}\right),
\end{equation*}
and
\begin{equation*}
0>\mathscr{P}_{1, 2N_{1}(\delta)+1}-\mathscr{P}_{2} \geq  -\exp\left({-\big(\frac{r+1}{r}\big)^{1/2}\frac{1}{\sqrt{\delta}|\log \delta|}}\right).
\end{equation*}
\end{enumerate}
\end{lemma}

\begin{proof}
We first remark that the following $O(\delta)$ and $O(\sqrt{\delta})$ are independent of $r$.

$(i)$ It follows from \eqref{2.4}--\eqref{B_{1}} that $\mathscr{P}_{1,2k}$ decreases to $\mathscr{P}_{1}$ and $\mathscr{P}_{1,2k+1}$ increases to $\mathscr{P}_{2}$. Hence,
\begin{equation}\label{mu 1 2k}
\mu_{1, 2k}=\frac{r_{1}}{r_{1}+\varepsilon-p_{1, 2k-1}}=\frac{1}{1+\delta-\mathscr{P}_{1, 2k-1}}\leq \frac{1}{1+\delta-\mathscr{P}_{2}} ,\quad\forall~ k\geq 1,
\end{equation}
and
\begin{equation}\label{mu 1 2k+1}
\mu_{1, 2k+1}=\frac{r_{2}}{r_{2}+\varepsilon+p_{1, 2k}}=\frac{r}{r+\delta+\mathscr{P}_{1, 2k}}\leq \frac{r}{r+\delta+\mathscr{P}_{1}},\quad\forall~ k\geq 0.
\end{equation}
For $k\geq m\geq 1$, by using \eqref{q_{1}}, \eqref{mu 1 2k}, and \eqref{mu 1 2k+1}, we have
\begin{align}\label{2.38}
q_{1, 2k}&=q_{1, 2m}\prod_{j=2m+1}^{2k}\mu_{1, j}\leq q_{1, 2m}\left(\frac{r}{r+\delta+\mathscr{P}_{1}}\right)^{k-m}
\left(\frac{1}{1+\delta-\mathscr{P}_{2}}\right)^{k-m}.
\end{align}
Since $\varepsilon$ is sufficiently small, it follows from \eqref{2.38}, \eqref{q_{1,2k}}, and $N(\delta)\simeq\Big(\frac{r}{r+1}\Big)^{1/2}\frac{1}{\sqrt{\delta}}$ that
\begin{align*}
&\sum_{k=N(\delta)}^{\infty}q_{1, 2k}\\
&\leq \sum_{k=N(\delta)}^{\infty}q_{1, 2N(\delta)}\left(\frac{r}{\big(r+\delta+\mathscr{P}_{1}\big)\big(1+\delta-\mathscr{P}_{2}\big)}\right)^{k-N(\delta)}\\
&\leq \left(\frac{r}{N(r+1)+r}+C\Big(\frac{r+1}{r}\delta\Big)^{1/2}\right)
\sum_{k=N(\delta)}^{\infty}\left(\frac{r}{\big(r+\delta+\mathscr{P}_{1}\big)\big(1+\delta-\mathscr{P}_{2}\big)}\right)^{k-N(\delta)}\\
&\leq C,
\end{align*}
and
\begin{equation*}
\left|\sum_{k=0}^{\infty}q_{1, 2k}-\sum_{k=0}^{N(\delta)-1}\Theta_{1,2k}(r)\right|\leq \sum_{k=0}^{N(\delta)-1}\left|q_{1,2k}-\Theta_{1,2k}(r)\right|+\sum_{k=N(\delta)}^{\infty}q_{1, 2k}\leq C.
\end{equation*}
Therefore, \eqref{2.22} is proved. Similarly, we have \eqref{2.22'}.

$(ii)$ It follows from \eqref{2.4}, \eqref{A_{1}}, and \eqref{B_{1}} that
\begin{align*}
\mathscr{P}_{1, 2k}-\mathscr{P}_{1}
&\geq\sqrt{\delta}\left(\Big(\frac{1}{4}\Big(\frac{r+1}{r}\Big)^{1/2}
+\frac{r+1}{r}O(\sqrt{\delta})\Big)A_{1}^{k}\right)^{-1}\\
&\geq 2\Big(\frac{r}{r+1}\delta\Big)^{1/2}A^{-k}_{1}.
\end{align*}
Similarly, we have
\begin{equation*}
\mathscr{P}_{1, 2k+1}-\mathscr{P}_{2}\leq -2\Big(\frac{r}{r+1}\delta\Big)^{1/2}A^{-k}_{2}.
\end{equation*}
Hence, \eqref{2.26} is proved.

$(iii)$ Now, suppose that $k\leq N(\delta)$. Combining
\begin{equation*}
A_{1}=1+2\mathscr{P}_{1}\frac{r+1}{r}+\frac{r+1}{r}O(\delta)\leq1+3\mathscr{P}_{1}\frac{r+1}{r}
\end{equation*}
and the inequality
$$(1+s)^{k}\leq 1+ks+\frac{1}{2}k^{2}s^{2}(1+s)^{k},\quad\forall~s>0,$$
we obtain
\begin{align}\label{2.39}
A_{1}^{k}&\leq\left(1+3\mathscr{P}_{1}\frac{r+1}{r}\right)^{k}\nonumber\\
&\leq1+k\Big(3\mathscr{P}_{1}\frac{r+1}{r}\Big)
+\frac{1}{2}k^{2}\Big(3\mathscr{P}_{1}\frac{r+1}{r}\Big)^{2}\left(1+3\mathscr{P}_{1}\frac{r+1}{r}\right)^{k}.
\end{align}
Using $$k\mathscr{P}_{1}\leq N\mathscr{P}_{1}\leq C\frac{r}{r+1}$$ and the fact that $(1+t)^{\frac{1}{t}}$ converges to $e$ as $t\rightarrow 0^{+}$, we have
$$\left(1+3\mathscr{P}_{1}\frac{r+1}{r}\right)^{k}\leq \exp\left({3k\mathscr{P}_{1}\frac{1+r}{r}}\right)\leq C.$$
Coming back to \eqref{2.39}, there exists some constant C independent of $k\leq N(\delta)$, $r$ and $\delta$, such that
\begin{equation*}
A_{1}^{k}\leq 1+C\frac{r+1}{r}k\mathscr{P}_{1}.
\end{equation*}
We then infer from \eqref{2.4} that
\begin{equation*}
\mathscr{P}_{1,2k}-\mathscr{P}_{1}
=\left(\left(1+O(\sqrt{\delta})+B_{1}\right)A_{1}^{k}-B_{1}\right)^{-1}\geq C\frac{r}{r+1}\frac{1}{k}.
\end{equation*}
Similarly, we get the second inequality of \eqref{2.30}.

$(iv)$ From the definition of $N_{1}(\delta)$, we have
\begin{equation*}
\log A_{1}^{N_{1}(\delta)}=N_{1}(\delta)\log A_{1}\geq \frac{N_{1}(\delta)}{2}(A_{1}-1)
\geq \Big(\frac{r+1}{r}\Big)^{1/2}\frac{1}{\sqrt{\delta}|\log \delta|}.
\end{equation*}
Therefore, we get
$$A_{1}^{N_{1}(\delta)}\geq \exp\left({\big(\frac{r+1}{r}\big)^{1/2}\frac{1}{\sqrt{\delta}|\log \delta|}}\right).$$
We then infer from \eqref{2.4} that
\begin{equation*}
0\leq \mathscr{P}_{1, 2N_{1}(\delta)}-\mathscr{P}_{1}
\leq \frac{1}{(A_{1}^{N_{1}(\delta)}-1)
B_{1}}\leq \exp\left({-\big(\frac{r+1}{r}\big)^{1/2}\frac{1}{\sqrt{\delta}|\log \delta|}}\right).
\end{equation*}
The second inequality in \eqref{lem3.2-4} can be proved similarly. Thus, the proof of Lemma \ref{lem3.2} is finished.
\end{proof}

\begin{remark}\label{rem2.31}
Replacing $r$ by $\frac{1}{r}$ in \eqref{def Theta}, and recalling that $\mathscr{P}_{2,2k}\in\mathfrak{B_{2}}$ and $\mathscr{P}_{2,2k+1}\in\mathfrak{B_{1}}$, we denote
$$\Theta_{2,2k}(r):=\frac{1}{k(r+1)+1},\quad\Theta_{2,2k+1}(r):=\frac{1}{(k+1)(r+1)}.$$
Then, we have
\begin{equation*}
\left|\sum_{k=0}^{\infty}q_{2, 2k}-\sum_{k=0}^{N(\delta)-1}\Theta_{2,2k}(r)\right|\leq C, \quad\sum_{k=N(\delta)}^{\infty}q_{2, 2k}\leq C,
\end{equation*}
and
\begin{equation*}
\left|\sum_{k=0}^{\infty}q_{2, 2k+1}-\sum_{k=0}^{N(\delta)-1}\Theta_{2,2k+1}(r)\right|\leq C, \quad \sum_{k=N(\delta)}^{\infty}q_{2, 2k+1}\leq C.
\end{equation*}
Moreover, a direct calculation gives
\begin{equation}\label{q1 2k 2k+1}
\sum_{k=0}^{\infty}q_{1, 2k}=\frac{1}{2}\frac{r}{r+1}|\log\varepsilon|+O(1),\quad
\sum_{k=0}^{\infty}q_{1, 2k+1}=\frac{1}{2}\frac{r}{r+1}|\log\varepsilon|+O(1),
\end{equation}
\begin{equation}\label{q2 2k 2k+1}
\sum_{k=0}^{\infty}q_{2, 2k}=\frac{1}{2}\frac{1}{r+1}|\log\varepsilon|+O(1),\quad
\sum_{k=0}^{\infty}q_{2, 2k+1}=\frac{1}{2}\frac{1}{r+1}|\log\varepsilon|+O(1).
\end{equation}
Here, $O(1)$ is independent of $\varepsilon$.
\end{remark}

Now, we are ready to prove Proposition \ref{prop Q}.
\subsection{Proof of Proposition \ref{prop Q}.}

\begin{proof}[Proof of Proposition \ref{prop Q}.]
From the definition of $Q_{1}$, we have
\begin{align*}
Q_{1}=\sum^{\infty}_{k=0}(q_{1,2k}-q_{1,2k+1})&=\sum^{N(\delta)-1}_{k=0}(q_{1,2k}-q_{1,2k+1})+\sum^{\infty}_{k=N(\delta)}(q_{1,2k}-q_{1,2k+1})\\
&=:Q_{1}^{N(\delta)}+Q_{1}^{R}.
\end{align*}
For $k\leq N(\delta)$, from \eqref{mu_{1}} and \eqref{y_{1,2k}}, we have
\begin{align}\label{mu_{1, 2k+1}'}
\mu_{1, 2k+1}&=\frac{r}{r+\delta+\mathscr{P}_{1, 2k}}\nonumber\\
&=\frac{k(r+1)+r}{(k+1)(r+1)}\left(1+\Big(r(r+1)\Big)^{-1/2}O(\sqrt{\delta})\right),
\end{align}
Then combining $\eqref{q_{1,2k}}$, \eqref{mu_{1, 2k+1}'}, and
$$\sum^{N(\delta)-1}_{k=0}\frac{1}{k+1}=\frac{1}{2}|\log\varepsilon|+O(1),$$
we conclude that the summations of the first $N(\delta)$-terms of $Q_{1}$ is
\begin{align*}
Q_{1}^{N(\delta)}&=\sum^{N(\delta)-1}_{k=0}q_{1,2k}(1-\mu_{1,2k+1})\\
&=\sum^{N(\delta)-1}_{k=0}\Theta_{1,2k}(r)\Theta_{2,2k+1}(r)+\Big(\frac{r+1}{r}\Big)^{1/2}O\left(|\log\varepsilon|\sqrt{\delta}\right).
\end{align*}
By using the decreasing property of $q_{1,j}$, \eqref{q_{1,2k}}, and $N(\delta)\simeq\Big(\frac{r}{r+1}\Big)^{1/2}\frac{1}{\sqrt{\delta}}$, we have
$$0<Q_{1}^{R}\leq q_{1,2N(\delta)}\leq \frac{Cr}{N(\delta)(r+1)+r},$$
which means that $Q_{1}^{R}$ converges to 0 as $\delta\rightarrow 0$.
Therefore, letting $\delta\rightarrow 0$, we have
\begin{align*}
Q_{1}
=\frac{r}{(r+1)^{2}}\sum^{\infty}_{k=1}\frac{1}{k\big(k-\frac{1}{r+1}\big)}
=-\frac{r}{r+1}\psi\Big(\frac{r}{r+1}\Big).
\end{align*}

Similarly,
$$Q_{2}=-\frac{1}{r+1}\psi\Big(\frac{1}{r+1}\Big).$$
Hence, it follows from \eqref{q1 2k 2k+1} and \eqref{q2 2k 2k+1} that
\begin{align*}
M&=Q_{2}\sum^{\infty}_{k=0}q_{1,2k}+Q_{1}\sum^{\infty}_{k=0}q_{2,2k+1}\\
&=-\frac{1}{2}\frac{r}{(r+1)^{2}}|\log\varepsilon|
\left(\psi\Big(\frac{1}{r+1}\Big)+\psi\Big(\frac{r}{r+1}\Big)\right)
\left(1+\frac{r+1}{r}O(|\log\delta|^{-1})\right).
\end{align*}
The proof of Proposition \ref{prop Q} is completed.
\end{proof}

\begin{remark}
For $r_{1}\geq r_{2}$, we have
\begin{equation*}
0<\frac{r_{2}}{r_{1}+r_{2}}\leq\frac{1}{2}\leq\frac{r_{1}}{r_{1}+r_{2}}<1.
\end{equation*}
Then
\begin{equation*}
-\frac{r_{1}+r_{2}}{r_{2}}\leq\psi\Big(\frac{r_{2}}{r_{1}+r_{2}}\Big)\leq1-\frac{r_{1}+r_{2}}{r_{2}}=-\frac{r_{1}}{r_{2}}.
\end{equation*}
Thus,
\begin{equation}\label{est Q1}
\frac{r_{1}}{r_{1}+r_{2}}\leq Q_{1}\leq1\leq\frac{2r_{1}}{r_{1}+r_{2}}.
\end{equation}
Similarly, we have
\begin{equation}\label{est Q2}
\frac{r_{2}}{r_{1}+r_{2}}\leq Q_{2}\leq\frac{2r_{2}}{r_{1}+r_{2}}.
\end{equation}
The estimates \eqref{est Q1} and \eqref{est Q2} are exactly \eqref{est Qi} in \cite{LY1}.
\end{remark}

\section{Proof of Proposition \ref{prop h}}\label{sec pf prop h}

\subsection{The outline of the proof of Proposition \ref{prop h}}

For $\textbf{x}=(x_{1}, x_{2}, x_{3})$, we denote $\tilde{\textbf{x}}:=\frac{\textbf{x}}{r_{1}}=(\tilde{x}_{1}, \tilde{x}_{2}, \tilde{x}_{3})$, and
$$\alpha_{k}:=\frac{q_{1, 2k}}{|\tilde{\textbf{x}}-{\pmb{\mathscr{P}}}_{1, 2k}|}-\frac{q_{1, 2k+1}}{|\tilde{\textbf{x}}-{\pmb{\mathscr{P}}}_{1, 2k+1}|},\quad\beta_{k}:=\frac{q_{2, 2k+1}}{|\tilde{\textbf{x}}-{\pmb{\mathscr{P}}}_{2, 2k+1}|}-\frac{q_{2, 2k}}{|\tilde{\textbf{x}}-{\pmb{\mathscr{P}}}_{2, 2k}|}.$$
By using Lemma \ref{lem2.1}, we have
\begin{align}\label{v(tilde{x})}
\big(-4\pi M\big)\Big(\frac{Q_{1}}{r_{2}}+\frac{Q_{2}}{r_{1}}\Big)^{-1}h(\textbf{x})
&=\frac{rQ_{2}}{Q_{1}+rQ_{2}}\sum\limits_{k=0}^{\infty}\alpha_{k}+\frac{rQ_{1}}{Q_{1}+rQ_{2}}\sum\limits_{k=0}^{\infty}\beta_{k}\nonumber\\
&=:\frac{rQ_{2}}{Q_{1}+rQ_{2}}v_{1}(\tilde{\textbf{x}})
+\frac{rQ_{1}}{Q_{1}+rQ_{2}}v_{2}(\tilde{\textbf{x}})\nonumber\\
&=:v(\tilde{\textbf{x}}).
\end{align}

From the definition of $v(\tilde{\bf{x}})$, in order to prove Proposition \ref{prop h}, it suffices to establish the asymptotic formula of $\nabla v(\tilde{\bf{x}})$ in $R_{\delta}$. We first give the estimates of $|\partial_{\tilde{x}_{2}}v(\tilde{\bf{x}})|$ and $|\partial_{\tilde{x}_{3}}v(\tilde{\bf{x}})|$, whose proof will be given in Subsection \ref{pf lem rho v} later.
\begin{lemma}\label{lem rho v}
For ${\bf{x}}\in R_{\delta}$, we have
\begin{equation*}
|\partial_{\tilde{x}_{2}} v(\tilde{\bf{x}})|+|\partial_{\tilde{x}_{3}}v(\tilde{\bf{x}})|\leq
C\frac{r+1}{r}\frac{1}{\rho(\tilde{\bf{x}})+\Big(\frac{r\delta}{r+1}\Big)^{1/2}}
\left(1+\log\left(1+\frac{r+1}{r}\frac{\rho(\tilde{\bf{x}})^{2}}{\delta}\right)\right)
\end{equation*}
for some constant C independent of $r$ and $\delta$.
\end{lemma}

For the estimate of $\partial_{\tilde{x}_{1}} v(\tilde{\textbf{x}})$, especially the term for $N_{0}(\delta)\leq k\leq N_{1}(\delta)$, is quite involved. In order to obtain the asymptotic formula of $\partial_{\tilde{x}_{1}} v(\tilde{\textbf{x}})$ in the narrow region $R_{\delta}$, we need to study the finer properties of the sequences $\mathscr{P}_{i,j}$ and $q_{i,j}$. The following Lemma is an adaption of Lemma 3.3 in \cite{KLY2}. Its proof is given in the Appendix.

\begin{lemma}\label{lem finer p q}
\begin{enumerate}[(i)]
\item\label{lem3.3-1}
If $N_{0}(\delta)\leq k \leq N_{1}(\delta)$, then
\begin{align}\label{-y_{1, 2k}+y_{1, 2k+2}}
&q_{1, 2k}\Big(\mathscr{P}_{1, 2k}-\mathscr{P}_{1, 2k+2}\Big)^{-\frac{1}{r+1}}\Big(\mathscr{P}_{1, 2k+3}-\mathscr{P}_{1, 2k+1}\Big)^{-\frac{r}{r+1}}\nonumber\\
=&\left(\frac{r}{r+1}+\frac{r+1}{r}O(|\log\delta|^{-1})\right)\Big(\mathscr{P}_{1, 2k}^{2}-\mathscr{P}_{1}^{2}\Big)^{-\frac{1}{2(r+1)}}\Big(\mathscr{P}_{1, 2k+1}^{2}-\mathscr{P}_{2}^{2}\Big)^{-\frac{r}{2(r+1)}},
\end{align}
where $O(|\log\delta|^{-1})$ is independent of $k$ and $r$.
\item\label{lem3.3-2}
There are positive constants $C_{1}$ and $C_{2}$ such that for all $k\geq N_{1}(\delta)$,
\begin{align}\label{q_{1, 2k}leq e}
q_{1, 2k}\leq C_{1}&\left(\frac{r}{r+\delta+\mathscr{P}_{1}}\right)^{k-N_{1}(\delta)}
\left(\frac{1}{1+\delta-\mathscr{P}_{2}}\right)^{k-N_{1}(\delta)}\nonumber\\
&\cdot\exp\left(-C_{2}\big(\frac{r}{r+1}\big)^{1/2}\frac{1}{\sqrt{\delta}|\log \delta|}\right).
\end{align}
\end{enumerate}
\end{lemma}

Consider the following two auxiliary functions:
\begin{align*}
f_{1}(t)&:=\left(|\tilde{\textbf{x}}-(t, 0, 0)|^{-1}-|\tilde{\textbf{x}}+(t, 0, 0)|^{-1}\right)\frac{1}{\sqrt{t^{2}-\mathscr{P}_{1}^{2}}},\\
f_{2}(t)&:=\left(|\tilde{\textbf{x}}-(t, 0, 0)|^{-1}-|\tilde{\textbf{x}}+(t, 0, 0)|^{-1}\right)\frac{1}{\sqrt{t^{2}-\mathscr{P}_{2}^{2}}}.
\end{align*}
Define
\begin{equation*}
v_{1}^{0}(\tilde{\textbf{x}}):=\int_{\mathscr{P}_{1}}^{1}f_{1}(t)\ dt,\quad v_{2}^{0}(\tilde{\textbf{x}}):=\int_{-1}^{\mathscr{P}_{2}}f_{2}(t)\ dt.
\end{equation*}
Here, $\pmb{\mathscr{P}}_{1}=(\mathscr{P}_{1}, 0, 0)$ and $\pmb{\mathscr{P}}_{2}=(\mathscr{P}_{2}, 0, 0)$ are the fixed points of combined reflection $R_{1}R_{2}$ and $R_{2}R_{1}$, respectively. We obtain the following two lemmas, whose proofs will be given in Subsection \ref{pf lem x1 v0} and \ref{pf lem x1 v}, respectively.
\begin{lemma}\label{lem x1 v0}
For ${\bf{x}}\in R_{\delta}$, we have
\begin{equation*}
\partial_{\tilde{x}_{1}}v_{1}^{0}(\tilde{\bf{x}})=\frac{2}{\frac{4r}{r+1}\delta+\rho(\tilde{\bf{x}})^{2}}\left(1+\frac{r+1}{r}O(|\log\delta|^{-1})\right)
\end{equation*}
and
\begin{equation*}
\partial_{\tilde{x}_{1}}v_{2}^{0}(\tilde{\bf{x}})=-\frac{2}{\frac{4r}{r+1}\delta+\rho(\tilde{\bf{x}})^{2}}\left(1+\frac{r+1}{r}O(|\log\delta|^{-1})\right).
\end{equation*}
\end{lemma}
Then
\begin{lemma}\label{lem x1 v}
For ${\bf{x}}\in R_{\delta}$, we have
\begin{align*}
\partial_{\tilde{x}_{1}}v(\tilde{\bf{x}})
&=\frac{rQ_{2}}{Q_{1}+rQ_{2}}\left(\frac{r}{r+1}+\frac{r+1}{r}O(|\log \delta|^{-1})\right)\partial_{\tilde{x}_{1}}v_{1}^{0}(\tilde{\bf{x}})\nonumber\\
&\quad-\frac{Q_{1}}{Q_{1}+rQ_{2}}\left(\frac{r}{r+1}+\frac{r+1}{r}O(|\log \delta|^{-1})\right)\partial_{\tilde{x}_{1}}v_{2}^{0}(\tilde{\bf{x}}).
\end{align*}
\end{lemma}

Now, we are ready to prove Proposition \ref{prop h}.
\begin{proof}[\bf Proof of Proposition \ref{prop h}.]
For ${\bf{x}}\in R_{\delta}$, by using Lemma \ref{lem rho v}, Lemma \ref{lem x1 v0}, and Lemma \ref{lem x1 v}, we have
$$\nabla v(\tilde{\bf{x}})=\frac{2}{\frac{4r}{r+1}\delta+\rho(\tilde{\bf{x}})^{2}}
\left(\left(\frac{r}{r+1}, 0, 0\right)+\frac{r+1}{r}O(|\log \delta|^{-1})\right).$$
Combining the definition of $v(\tilde{\textbf{x}})$, $r$, $\delta$, and Proposition \ref{prop Q}, we obtain
\begin{align*}
\nabla h(\textbf{x})&=\frac{-1}{4\pi Mr_{1}}\Big(\frac{Q_{1}}{r_{2}}+\frac{Q_{2}}{r_{1}}\Big)\nabla v(\tilde{\bf{x}})\\
&=\frac{-(r_{1}+r_{2})}{4\pi r_{1}r_{2}|\log\varepsilon|}\frac{1}{\varepsilon+\frac{1}{4}\left(\frac{1}{r_{1}}+\frac{1}{r_{2}}\right)\rho({\bf{x}})^{2}}\left({\bf{n}}
+\Big(\frac{r_{1}+r_{2}}{r_{2}}\Big)^{2}O(|\log\varepsilon|^{-1})\right).
\end{align*}
Then Proposition \ref{prop h} is proved.
\end{proof}

\subsection{Proof of Lemma \ref{lem rho v}}\label{pf lem rho v}

We first observe that if $\textbf{x}=(x_{1}, x_{2}, x_{3})\in R_{\delta}$, then
$$|\tilde{x}_{1}|\leq 1+\delta-\sqrt{1-\rho(\tilde{\textbf{x}})^{2}},\quad \rho(\tilde{\textbf{x}})\leq |\log\delta|^{-2}.$$
Hence,
\begin{align}\label{|tilde{x}|}
|\tilde{x}_{1}|\leq \delta+\rho(\tilde{\textbf{x}})^{2}.
\end{align}

Using the notation
$$\tilde{\rho}:=\rho(\tilde{\textbf{x}}),$$
$v(\tilde{\textbf{x}})$ can be expressed as
\begin{equation*}
v(\tilde{\textbf{x}})=\frac{rQ_{2}}{Q_{1}+rQ_{2}}(v_{1, 1}+v_{1, 2})
+\frac{rQ_{1}}{Q_{1}+rQ_{2}}(v_{2, 1}+v_{2, 2}),
\end{equation*}
where
\begin{align*}
v_{1, 1}&=\sum\limits_{k=0}^{\infty}q_{1, 2k}\big(s_{1}^{-1/2}-t_{1}^{-1/2}\big),\quad v_{1, 2}=\sum\limits_{k=0}^{\infty}q_{1, 2k}t_{1}^{-1/2}-\sum\limits_{k=0}^{\infty}q_{1, 2k+1}t_{2}^{-1/2},\\
v_{2, 1}&=\sum\limits_{k=0}^{\infty}q_{2, 2k+1}\big({\tilde{t}}_{1}^{-1/2}-s_{2}^{-1/2}\big),\quad v_{2, 2}=\sum\limits_{k=0}^{\infty}q_{2, 2k+1}s_{2}^{-1/2}-\sum\limits_{k=0}^{\infty}q_{2, 2k}{\tilde{s}}_{1}^{-1/2},
\end{align*}
\begin{align*}
&s_{1}=(\tilde{x}_{1}-\mathscr{P}_{1, 2k})^{2}+\tilde{\rho}^{2},\quad \tilde{s}_{1}=(\tilde{x}_{1}-\mathscr{P}_{2, 2k})^{2}+\tilde{\rho}^{2},\quad s_{2}=(\tilde{x}_{1}+\mathscr{P}_{2, 2k+1})^{2}+\tilde{\rho}^{2},\\
&t_{1}=(\tilde{x}_{1}+\mathscr{P}_{1, 2k})^{2}+\tilde{\rho}^{2},\quad \tilde{t}_{1}=(\tilde{x}_{1}-\mathscr{P}_{2, 2k+1})^{2}+\tilde{\rho}^{2},\quad t_{2}=(\tilde{x}_{1}-\mathscr{P}_{1, 2k+1})^{2}+\tilde{\rho}^{2}.
\end{align*}
Therefore, in order to estimate $|\partial_{\tilde{x}_{2}} v|$ and $|\partial_{\tilde{x}_{3}} v|$, it suffices to estimate $|\partial_{\tilde{\rho}}v|$. We shall divide the rest of the proof into two steps.

\noindent{\bf Step 1. Estimates of $|\partial_{\tilde{\rho}}v_{1,1}|$ and $|\partial_{\tilde{\rho}}v_{2,1}|$.}
Notice that
\begin{align*}
\partial_{\tilde{\rho}}v_{1, 1}&=\sum\limits_{k=0}^{\infty}q_{1, 2k}\tilde{\rho}\big(t_{1}^{-\frac{3}{2}}-s_{1}^{-\frac{3}{2}}\big)\\
&=3\tilde{\rho}\sum\limits_{k=0}^{\infty}q_{1, 2k}\int_{-\tilde{x}_{1}}^{\tilde{x}_{1}}\big(t-\mathscr{P}_{1, 2k}\big)\left(\left(t-\mathscr{P}_{1, 2k}\right)^{2}+\tilde{\rho}^{2}\right)^{-5/2}\ dt.
\end{align*}
Therefore, we have
\begin{equation*}
|\partial_{\tilde{\rho}}v_{1, 1}|\leq 3\tilde{\rho}\sum\limits_{k=0}^{\infty}q_{1, 2k}\int_{-\tilde{x}_{1}}^{\tilde{x}_{1}}\left(\left(t-\mathscr{P}_{1, 2k}\right)^{2}+\tilde{\rho}^{2}\right)^{-2}\ dt.
\end{equation*}
By \eqref{|tilde{x}|}, there exists some constant $C$ such that
\begin{equation*}
\left(t-\mathscr{P}_{1, 2k}\right)^{2}+\tilde{\rho}^{2}\geq C(\tilde{\rho}^{2}+\mathscr{P}_{1, 2k}^{2}),\quad\forall~k\geq 0.
\end{equation*}
It then follows that
\begin{align*}
|\partial_{\tilde{\rho}}v_{1, 1}|&\leq C\tilde{\rho}\sum\limits_{k=0}^{\infty}q_{1, 2k}\int_{-\tilde{x}_{1}}^{\tilde{x}_{1}}\left(\tilde{\rho}^{2}+\mathscr{P}_{1, 2k}^{2}\right)^{-2}\ dt\\
&\leq C\frac{r+1}{r}\tilde{\rho}\sum\limits_{k=0}^{\infty}q_{1, 2k}\left(\tilde{\rho}^{2}+\mathscr{P}_{1, 2k}^{2}\right)^{-1}.
\end{align*}
For $k\leq N(\delta)-1$, we obtain from \eqref{q_{1,2k}} that
\begin{equation}\label{q_{1,2k}leq}
q_{1,2k}\leq C \frac{1}{k+1}, \quad q_{1,2k+1}\leq C \frac{1}{k+1}.
\end{equation}
It is easy see from Lemma \ref{lem3.2} \eqref{lem3.2-3} that
\begin{equation}\label{y_{1,2k}leq}
\mathscr{P}_{1,2k}\geq C\Theta_{1,2k+1}(r), \quad
\mathscr{P}_{1,2k+1}\leq -C\Theta_{1,2k+1}(r).
\end{equation}
Then, by using \eqref{q_{1,2k}leq} and \eqref{y_{1,2k}leq}, we have
\begin{align}\label{I}
&\tilde{\rho}\sum\limits_{k=0}^{N(\delta)-1}q_{1, 2k}\left(\tilde{\rho}^{2}+\mathscr{P}_{1, 2k}^{2}\right)^{-1}\nonumber\\
&\leq C\sum\limits_{k=1}^{N(\delta)}k\tilde{\rho}\left(\tilde{\rho}^{2}k^{2}+\Big(\frac{r}{r+1}\Big)^{2}\right)^{-1}\nonumber\\
&\leq C\left(\int_{1}^{N(\delta)}\tilde{\rho}s\left(\tilde{\rho}^{2}s^{2}+\Big(\frac{r}{r+1}\Big)^{2}\right)^{-1}\ ds
+\frac{r+1}{r}\right)\nonumber\\
&\leq C\left(\tilde{\rho}+\Big(\frac{r\delta}{r+1}\Big)^{1/2}\right)^{-1}\left(1+\log\left(1+\frac{r+1}{r}\frac{\tilde{\rho}^{2}}{\delta}\right)\right).
\end{align}
In view of \eqref{2.22} and \eqref{2.30},
\begin{align*}
\tilde{\rho}\sum\limits_{k=N(\delta)}^{\infty}q_{1, 2k}\left(\tilde{\rho}^{2}+\mathscr{P}_{1, 2k}^{2}\right)^{-1}
&\leq \tilde{\rho}\sum\limits_{k=N(\delta)}^{\infty}q_{1, 2k}\left(\tilde{\rho}^{2}+\frac{r\delta}{r+1}\right)^{-1}\\
&\leq C\left(\tilde{\rho}+\Big(\frac{r\delta}{r+1}\Big)^{1/2}\right)^{-1}.
\end{align*}
Thus,
\begin{equation}\label{est v11}
|\partial_{\tilde{\rho}}v_{1, 1}|
\leq C\frac{r+1}{r}\left(\tilde{\rho}+\Big(\frac{r\delta}{r+1}\Big)^{1/2}\right)^{-1}\left(1+\log\left(1+\frac{r+1}{r}\frac{\tilde{\rho}^{2}}{\delta}\right)\right).
\end{equation}
Similarly, $|\partial_{\tilde{\rho}}v_{2, 1}|$ is also bounded by the right-hand side of \eqref{est v11}.

\noindent{\bf Step 2. Estimates of $|\partial_{\tilde{\rho}}v_{1, 2}|$ and $|\partial_{\tilde{\rho}}v_{2, 2}|$.}
By a direct calculation, we have
\begin{align}\label{ineq v12}
\partial_{\tilde{\rho}}v_{1, 2}&=\sum\limits_{k=0}^{\infty}q_{1, 2k}\tilde{\rho}\left(\mu_{1, 2k+1} t_{2}^{-\frac{3}{2}}-t_{1}^{-\frac{3}{2}}\right)\nonumber\\
&=\sum\limits_{k=0}^{N(\delta)-1}
+\sum\limits_{k=N(\delta)}^{\infty}q_{1, 2k}\tilde{\rho}\left(\mu_{1, 2k+1}t_{2}^{-\frac{3}{2}}-t_{1}^{-\frac{3}{2}}\right)\nonumber\\
&=:\mbox{I}+\mbox{II}.
\end{align}
Notice that
\begin{equation*}
\mbox{I}
=\sum\limits_{k=0}^{N(\delta)-1}q_{1, 2k}\tilde{\rho}t_{2}^{-\frac{3}{2}}\left(\mu_{1, 2k+1}-1\right)+\sum\limits_{k=0}^{N(\delta)-1}q_{1, 2k}\tilde{\rho}\big(t_{2}^{-\frac{3}{2}}-t_{1}^{-\frac{3}{2}}\big)=:\mbox{I}_{1}+\mbox{I}_{2}.
\end{equation*}

Recall \eqref{mu_{1, 2k+1}'} implies that
\begin{equation}\label{|1-mu_{1, 2k}|}
|1-\mu_{1, 2k+1}|\leq C\Theta_{2,2k+1}(r) \quad \text{for}~k\leq N(\delta)-1.
\end{equation}
For $k\leq N(\delta)-1$, it results from \eqref{|tilde{x}|} and \eqref{y_{1,2k}leq} that
\begin{equation}\label{(tilde{x}-y_{1, 2k+1})^{2}}
\left(\tilde{x}_{1}-\mathscr{P}_{1, 2k+1}\right)^{2}+\tilde{\rho}^{2}
\geq C(\tilde{\rho}^{2}+\mathscr{P}_{1, 2k+1}^{2})\geq C\left(\tilde{\rho}^{2}+(\Theta_{1,2k+1}(r))^{2}\right),
\end{equation}
and
\begin{equation}\label{(tilde{x}+y_{1, 2k})^{2}}
\left(\tilde{x}_{1}+\mathscr{P}_{1, 2k}\right)^{2}+\tilde{\rho}^{2}
\geq C(\tilde{\rho}^{2}+\mathscr{P}_{1, 2k}^{2})\geq C\left(\tilde{\rho}^{2}+(\Theta_{1,2k+1}(r))^{2}\right).
\end{equation}
Hence, we have
\begin{equation*}
|~\mbox{I}_{1}~|\leq C\frac{r+1}{r}\sum\limits_{k=0}^{N(\delta)-1}
\frac{q_{1, 2k}\tilde{\rho}}{\mathscr{P}_{1, 2k+1}^{2}+\tilde{\rho}^{2}}.
\end{equation*}
Similar to \eqref{I}, we have
\begin{equation}\label{est I1}
|~\mbox{I}_{1}~|\leq C\frac{r+1}{r}\left(\tilde{\rho}+\Big(\frac{r\delta}{r+1}\Big)^{1/2}\right)^{-1}\left(1+\log\left(1+\frac{r+1}{r}\frac{\tilde{\rho}^{2}}{\delta}\right)\right).
\end{equation}
A direct calculation gives
\begin{equation}\label{3.12}
\left|t_{2}^{-\frac{3}{2}}-t_{1}^{-\frac{3}{2}}\right|\leq |\mathscr{P}_{1, 2k}+\mathscr{P}_{1, 2k+1}|
\left(t_{2}^{-1/2}t_{1}^{-3/2}+t_{2}^{-3/2}t_{1}^{-1/2}+t_{1}^{-1}t_{2}^{-1}\right).
\end{equation}
By \eqref{(tilde{x}-y_{1, 2k+1})^{2}} and \eqref{(tilde{x}+y_{1, 2k})^{2}},
\begin{equation*}
|~\mbox{I}_{2}~|\leq C\sum\limits_{k=0}^{N(\delta)-1}q_{1, 2k}\tilde{\rho}\frac{\big|\mathscr{P}_{1, 2k}+\mathscr{P}_{1, 2k+1}\big|}{\left((\Theta_{1,2k+1}(r))^{2}+\tilde{\rho}^{2}\right)^{2}}.
\end{equation*}
By the definition of $N$, for $k\leq N(\delta)$, there is a constant C independent of $r$ and $\delta$ such that
\begin{equation*}
\sqrt{\delta}\leq C\Big(\frac{r}{r+1}\Big)^{1/2}\frac{1}{k}.
\end{equation*}
By \eqref{p_{1,2k}p_{1,2k+1}} and \eqref{y_{1,2k}}, we have for $k\leq N(\delta)$,
\begin{equation}\label{y_{1, 2k}+y_{1, 2k+1}}
\left|\mathscr{P}_{1, 2k}+\mathscr{P}_{1, 2k+1}\right|=\left|\frac{\delta^{2}+2\delta r-\mathscr{P}_{1, 2k}^{2}}{r+\delta+\mathscr{P}_{1, 2k}}\right|\leq C\frac{r}{r+1}\frac{1}{k^{2}}.
\end{equation}
Hence,
\begin{equation*}
|~\mbox{I}_{2}~|\leq C\frac{r+1}{r}\sum\limits_{k=0}^{N(\delta)-1}q_{1, 2k}\tilde{\rho}\frac{1}{(\Theta_{1,2k+1}(r))^{2}+\tilde{\rho}^{2}}.
\end{equation*}
Similar to \eqref{I}, $|~\mbox{I}_{2}~|$ is also bounded by the right-hand side of \eqref{est I1}.

Notice that
\begin{equation*}
\mbox{II}
=\sum\limits_{k=N(\delta)}^{\infty}q_{1, 2k}\tilde{\rho}t_{2}^{-\frac{3}{2}}\left(\mu_{1, 2k+1}-1\right)+\sum\limits_{k=N(\delta)}^{\infty}q_{1, 2k}\tilde{\rho}\left(t_{2}^{-\frac{3}{2}}-t_{1}^{-\frac{3}{2}}\right)=:\mbox{II}_{1}+\mbox{II}_{2}.
\end{equation*}
For $k\geq N(\delta)$, using \eqref{y_{1,2k}} and the fact that the sequence $\mathscr{P}_{1, 2k}$ is decreasing to $\mathscr{P}_{1}$, we have
\begin{equation}\label{eq p 1 2k}
\mathscr{P}_{1, 2k}\simeq C\Big(\frac{r\delta}{r+1}\Big)^{1/2}.
\end{equation}
Then
\begin{equation}\label{|1-mu_{1, 2k}|'}
|1-\mu_{1, 2k+1}|=\left|\frac{\delta+\mathscr{P}_{1, 2k}}{r+\delta+\mathscr{P}_{1, 2k}}\right|
\leq C\Big(\frac{r+1}{r}\delta\Big)^{1/2}.
\end{equation}
On the other hand, recalling the fact that $\mathscr{P}_{1, 2k+1}$ is increasing to $\mathscr{P}_{2}$, it follows from \eqref{|tilde{x}|} and \eqref{p_{1}} that for all $k\geq 0$,
\begin{equation}\label{(tilde{x}-y_{1, 2k+1})'^{2}}
\left(\tilde{x}_{1}-\mathscr{P}_{1, 2k+1}\right)^{2}+\tilde{\rho}^{2}
\geq C(\tilde{\rho}^{2}+\mathscr{P}_{1, 2k+1}^{2})\geq C\left(\tilde{\rho}^{2}+\frac{r\delta}{r+1}\right).
\end{equation}
Similarly,
\begin{equation}\label{(tilde{x}+y_{1, 2k})'^{2}}
\left(\tilde{x}_{1}+\mathscr{P}_{1, 2k}\right)^{2}+\tilde{\rho}^{2}
\geq C(\tilde{\rho}^{2}+\mathscr{P}_{1, 2k}^{2})\geq C\left(\tilde{\rho}^{2}+\frac{r\delta}{r+1}\right),\quad\forall~k\geq 0.
\end{equation}
It follows from \eqref{2.22}, \eqref{|1-mu_{1, 2k}|'}, and \eqref{(tilde{x}-y_{1, 2k+1})'^{2}} that
\begin{equation*}
|\mbox{II}_{1}|=\sum\limits_{k=N(\delta)}^{\infty}q_{1, 2k}\tilde{\rho}t_{2}^{-\frac{3}{2}}|\mu_{1, 2k+1}-1|
\leq C\frac{r+1}{r}\left(\tilde{\rho}+\Big(\frac{r}{r+1}\delta\Big)^{1/2}\right)^{-1}.
\end{equation*}
For $k\geq N(\delta)$, by using \eqref{eq p 1 2k}, we have
\begin{align}\label{|y_{1, 2k+1}+y_{1, 2k}|'}
|\mathscr{P}_{1, 2k+1}+\mathscr{P}_{1, 2k}|&=\left|\delta^{2}+2\delta r-\mathscr{P}_{1, 2k}^{2}\right|(r+\delta+\mathscr{P}_{1, 2k})^{-1}\nonumber\\
&\leq C\delta.
\end{align}
Hence, we obtain from \eqref{2.22}, \eqref{3.12}, \eqref{(tilde{x}-y_{1, 2k+1})'^{2}}--\eqref{|y_{1, 2k+1}+y_{1, 2k}|'}  that
\begin{align*}
|~\mbox{II}_{2}~|&\leq \sum\limits_{k=N(\delta)}^{\infty}q_{1, 2k}\tilde{\rho}|\mathscr{P}_{1, 2k}+\mathscr{P}_{1, 2k+1}|\left(\frac{r\delta}{r+1}+\tilde{\rho}^{2}\right)^{-2}\\
&\leq C\frac{r+1}{r}\left(\Big(\frac{r\delta}{r+1}\Big)^{1/2}+\tilde{\rho}\right)^{-1}.
\end{align*}
Coming back to \eqref{ineq v12}, we have
\begin{equation*}
|\partial_{\tilde{\rho}}v_{1, 2}|\leq C\frac{r+1}{r}\left(\tilde{\rho}+\Big(\frac{r\delta}{r+1}\Big)^{1/2}\right)^{-1}
\left(1+\log\left(1+\frac{r+1}{r}\frac{\tilde{\rho}^{2}}{\delta}\right)\right).
\end{equation*}
Similarly,
\begin{equation*}
|\partial_{\tilde{\rho}}v_{2, 2}| \leq C\frac{r+1}{r}\left(\tilde{\rho}+\Big(\frac{r\delta}{r+1}\Big)^{1/2}\right)^{-1}
\left(1+\log\left(1+\frac{r+1}{r}\frac{\tilde{\rho}^{2}}{\delta}\right)\right).
\end{equation*}
Lemma \ref{lem rho v} is proved.

\subsection{Proof of Lemma \ref{lem x1 v0}}\label{pf lem x1 v0}
By the definition of $v_{1}^{0}$, we have
\begin{equation*}
\partial_{\tilde{x}_{1}}v_{1}^{0}=\int_{\mathscr{P}_{1}}^{|\log\delta|^{-1}}+\int_{|\log\delta|^{-1}}^{1}\partial_{\tilde{x}_{1}}f_{1}(t)dt=:\mbox{J}_{1}+\mbox{J}_{2}.
\end{equation*}

If $|\log\delta|^{-1}\leq t\leq1$, then for all $\textbf{x}=(x_{1}, x_{2}, x_{3})\in R_{\delta}$, we have
$$|\tilde{\textbf{x}}\pm(t, 0, 0)|\geq C|t|$$
for some constant $C$. Since $\mathscr{P}_{1}=O(\sqrt{\delta})$, we have $\sqrt{t^{2}-\mathscr{P}_{1}^{2}}\geq C|t|$. Thus, we have
\begin{equation}\label{|I|leq C|log delta'|^{2}}
|\mbox{J}_{2}|\leq C\int_{|\log\delta|^{-1}}^{1}\frac{1}{t^{3}}\ dt\leq C|\log\delta|^{2}.
\end{equation}

Suppose now that $\mathscr{P}_{1}\leq t\leq|\log\delta|^{-1}$. Using \eqref{|tilde{x}|} and the fact that
$$\mathscr{P}_{1}^{2}\geq\frac{r\delta}{r+1}$$
again, we can see that for all $\textbf{x}=(x_{1}, x_{2}, x_{3})\in R_{\delta}$, there exists some constant $C$ independent of $r$ and $\delta$, such that
\begin{equation}\label{|t tilde{x}|}
|t\tilde{x}_{1}|\leq |t|(\delta+\tilde{\rho}^{2})\leq C\frac{r+1}{r}\frac{1}{|\log\delta|}(t^{2}+\tilde{\rho}^{2}),
\end{equation}
and
\begin{equation}\label{|tilde{x}|^{2}}
|\tilde{x}_{1}|^{2}\leq (\delta+\tilde{\rho}^{2})^{2}\leq C\frac{r+1}{r}\frac{1}{|\log\delta|}(t^{2}+\tilde{\rho}^{2}).
\end{equation}
Thus, we have
\begin{align}\label{|textbf{x}pm(t, 0, 0)|^{3}}
|\tilde{\textbf{x}}\pm(t, 0, 0)|^{-3}&=((\tilde{x}_{1}\pm t)^{2}+\tilde{\rho}^{2})^{-3/2}\nonumber\\
&=(t^{2}+\tilde{\rho}^{2})^{-3/2}\left(1+\frac{r+1}{r}O(|\log\delta|^{-1})\right).
\end{align}

From the mean value property, we have
\begin{align*}
&\left|-|\tilde{\textbf{x}}-(t, 0, 0)|^{-3}+|\tilde{\textbf{x}}+(t, 0, 0)|^{-3}\right|\\
&=\left|\Big((t^{2}+\tilde{x}_{1}^{2}+\tilde{\rho}^{2})-2\tilde{x}_{1}t\Big)^{-3/2}-\Big((t^{2}+\tilde{x}_{1}^{2}+\tilde{\rho}^{2})+2\tilde{x}_{1}t\Big)^{-3/2}\right|\\
&\leq6|\tilde{x}_{1}t|\Big((t^{2}+\tilde{x}_{1}^{2}+\tilde{\rho}^{2})-|2\tilde{x}_{1}t|\Big)^{-5/2}.
\end{align*}
It then follows from \eqref{|t tilde{x}|} and \eqref{|tilde{x}|^{2}} that
\begin{equation}\label{1'}
\left|\tilde{x}_{1}\left(-|\tilde{\textbf{x}}-(t, 0, 0)|^{-3}+|\tilde{\textbf{x}}+(t, 0, 0)|^{-3}\right)\right|
\leq C\frac{r+1}{r}|\log\delta|^{-1}|t|(t^{2}+\tilde{\rho}^{2})^{-3/2}
\end{equation}
for some constant $C>0$ independent of $r$ and $\delta$.

In view of \eqref{|textbf{x}pm(t, 0, 0)|^{3}} and \eqref{1'}, we have
\begin{align*}
&\partial_{\tilde{x}_{1}}\left(|\tilde{\textbf{x}}-(t, 0, 0)|^{-1}-|\tilde{\textbf{x}}+(t, 0, 0)|^{-1}\right)\\
&=t\left(|\tilde{\textbf{x}}+(t, 0, 0)|^{-3}+|\tilde{\textbf{x}}-(t, 0, 0)|^{-3}\right)
-\tilde{x}_{1}\left(|\tilde{\textbf{x}}-(t, 0, 0)|^{-3}-|\tilde{\textbf{x}}+(t, 0, 0)|^{-3}\right)\\
&=t(t^{2}+\tilde{\rho}^{2})^{-3/2}\left(2+\frac{r+1}{r}O(|\log\delta|^{-1})\right).
\end{align*}
Combining \eqref{p_{1}} and $\tilde{\rho}\leq |\log\delta|^{-2}$, we have
\begin{align}\label{property J2}
\mbox{J}_{1}&=\left(2+\frac{r+1}{r}O(|\log\delta|^{-1})\right)\int_{\mathscr{P}_{1}}^{|\log\delta|^{-1}}t (t^{2}+\tilde{\rho}^{2})^{-3/2}\Big(t^{2}-\mathscr{P}_{1}^{2}\Big)^{-1/2}\ dt\nonumber\\
&=\left(2+\frac{r+1}{r}O(|\log\delta|^{-1})\right)\Big(\mathscr{P}_{1}^{2}+\tilde{\rho}^{2}\Big)^{-1}
\left(\frac{|\log\delta|^{-2}-\mathscr{P}_{1}^{2}}{|\log\delta|^{-2}+\tilde{\rho}^{2}}\right)^{1/2}\nonumber\\
&=\frac{2}{\frac{4r}{r+1}\delta+\tilde{\rho}^{2}}\left(1+\frac{r+1}{r}O(|\log\delta|^{-1})\right).
\end{align}
By using \eqref{|I|leq C|log delta'|^{2}} and \eqref{property J2}, we get
\begin{equation*}
\partial_{\tilde{x}_{1}}v_{1}^{0}=\frac{2}{\frac{4r}{r+1}\delta+\tilde{\rho}^{2}}\left(1+\frac{r+1}{r}O(|\log\delta|^{-1})\right).
\end{equation*}
Similarly,
\begin{align*}
\partial_{\tilde{x}_{1}}v_{2}^{0}=-\frac{2}{\frac{4r}{r+1}\delta+\tilde{\rho}^{2}}\left(1+\frac{r+1}{r}O(|\log\delta|^{-1})\right).
\end{align*}
This completes the proof of Lemma \ref{lem x1 v0}.

\subsection{Proof of Lemma \ref{lem x1 v}}\label{pf lem x1 v}
From the definitions of $v_{1}$ and $v_{1}^{0}$, we have
\begin{equation*}
\partial_{\tilde{x}_{1}}v_{1}=\sum\limits_{k=0}^{N_{0}(\delta)-1}+\sum\limits_{k=N_{0}(\delta)}^{N_{1}(\delta)-1}
+\sum\limits_{k=N_{1}(\delta)}^{\infty}\partial_{\tilde{x}_{1}}\alpha_{k}=:S_{1}(\tilde{\textbf{x}})+S_{2}(\tilde{\textbf{x}})+S_{3}(\tilde{\textbf{x}})
\end{equation*}
and
\begin{equation*}
\partial_{\tilde{x}_{1}}v_{1}^{0}=\int_{\mathscr{P}_{1, 2N_{0}(\delta)}}^{1}+\int_{\mathscr{P}_{1, 2N_{1}(\delta)}}^{\mathscr{P}_{1, 2N_{0}(\delta)}}+\int_{\mathscr{P}_{1}}^{\mathscr{P}_{1, 2N_{1}(\delta)}}\partial_{\tilde{x}_{1}}f_{1}(t)\ dt=:\mbox{T}_{1}+\mbox{T}_{2}+\mbox{T}_{3}.
\end{equation*}
The rest of the proof is divided into four steps.

\noindent{\bf Step 1. Estimates of $S_{1}$, $S_{2}$, $S_{3}$, $\mbox{T}_{1}$, and $\mbox{T}_{3}$.}
By \eqref{|tilde{x}|}, one can see that for all $\textbf{x}\in R_{\delta}$, there is a constant $C>0$ independent of $k$ such that
$$|\tilde{\textbf{x}}-{\pmb{\mathscr{P}}}_{1, 2k}|\geq C |\mathscr{P}_{1, 2k}|, \quad|\tilde{\textbf{x}}-{\pmb{\mathscr{P}}}_{1, 2k+1}|\geq C |\mathscr{P}_{1, 2k+1}|.$$
So we have from \eqref{q_{1,2k}leq} and
\eqref{y_{1,2k}leq} that
\begin{equation*}
|S_{1}|\leq\sum\limits_{k=0}^{N_{0}(\delta)-1}|\partial_{\tilde{x}_{1}}\alpha_{k}|\leq C \left(\frac{r+1}{r}\right)^{2}|\log \delta|^{2}.
\end{equation*}
We use the fact that $\mathscr{P}_{1, 2k}$ is decreasing to $\mathscr{P}_{1}$, $\mathscr{P}_{1, 2k+1}$ is increasing to $\mathscr{P}_{2}$, and \eqref{q_{1, 2k}leq e} again to conclude that
\begin{align*}
|S_{3}|&\leq\sum\limits_{k=N_{1}(\delta)}^{\infty}|\partial_{\tilde{x}_{1}}\alpha_{k}|\\
&\leq C\frac{r+1}{r}\frac{1}{\delta}\sum\limits_{k=N_{1}(\delta)}^{\infty}q_{1, 2k}\\
&\leq C\frac{r+1}{r}\frac{1}{\delta}\sum\limits_{k=N_{1}(\delta)}^{\infty}\left(\frac{r}{\big(r+\delta+\mathscr{P}_{1}\big)\big(1+\delta-\mathscr{P}_{2}\big)}\right)^{k-N_{1}(\delta)}
e^{-C_{2}\big(\frac{r}{r+1}\big)^{1/2}
\frac{1}{|\sqrt{\delta}\log \delta|}}\\
&\leq C.
\end{align*}
Similarly, we have
\begin{equation*}
|\mbox{T}_{1}|\leq\int_{\mathscr{P}_{1, 2N_{0}(\delta)}}^{1}|\partial_{\tilde{x}_{1}}f_{1}(t)|\ dt\leq \int_{C\frac{r}{r+1}|\log \delta|^{-1}}^{1}\frac{1}{t^{3}}\ dt\leq C\left(\frac{r+1}{r}\right)^{2}|\log \delta|^{2}.
\end{equation*}
By Lemma \ref{lem3.2} \eqref{lem3.2-4} and the definition of $f_{1}$,
\begin{align*}
|\mbox{T}_{3}|\leq\int_{\mathscr{P}_{1}}^{\mathscr{P}_{1, 2N_{1}(\delta)}}|\partial_{\tilde{x}_{1}}f_{1}(t)|\ dt&\leq C\left(\frac{r+1}{r}\right)^{\frac{5}{4}}\int_{\mathscr{P}_{1}}^{\mathscr{P}_{1, 2N_{1}(\delta)}}
\delta^{-\frac{5}{4}}\frac{1}{\sqrt{t-\mathscr{P}_{1}}}\ dt\\
&\leq C.
\end{align*}
Thus, we have showed that
\begin{equation*}
|S_{1}|+|S_{3}|+|\mbox{T}_{1}|+|\mbox{T}_{3}|\leq C\left(\frac{r+1}{r}\right)^{2}|\log \delta|^{2}.
\end{equation*}

We set
\begin{align*}
S_{2}(\tilde{\textbf{x}})
&=\sum\limits_{k=N_{0}(\delta)}^{N_{1}(\delta)-1}\partial_{\tilde{x}_{1}}\left(\frac{q_{1, 2k}}{|\tilde{\textbf{x}}-{\pmb{\mathscr{P}}}_{1, 2k}|}-\frac{q_{1, 2k}}{|\tilde{\textbf{x}}+{\pmb{\mathscr{P}}}_{1, 2k}|}\right)\\
&\quad+\sum\limits_{k=N_{0}(\delta)}^{N_{1}(\delta)-1}\partial_{\tilde{x}_{1}}\left(\frac{q_{1, 2k}}{|\tilde{\textbf{x}}+{\pmb{\mathscr{P}}}_{1, 2k}|}-\frac{q_{1, 2k+1}}{|\tilde{\textbf{x}}-{\pmb{\mathscr{P}}}_{1, 2k+1}|}\right)\\
&=:S_{2, 1}(\tilde{\textbf{x}})+S_{2, 2}(\tilde{\textbf{x}}),
\end{align*}
and
\begin{equation*}
\widetilde{S}_{2, 1}(\tilde{\textbf{x}}):=\sum\limits_{k=N_{0}(\delta)}^{N_{1}(\delta)-1}\partial_{\tilde{x}_{1}}f_{1}(\mathscr{P}_{1,2k})(\mathscr{P}_{1, 2k}-\mathscr{P}_{1, 2k+2}).
\end{equation*}
We claim that
\begin{equation}\label{S_{2, 1}}
S_{2, 1}(\tilde{\textbf{x}})=\left(\frac{r}{r+1}+\frac{r+1}{r}O(|\log \delta|^{-1})\right)\widetilde{S}_{2, 1}(\tilde{\textbf{x}}),
\end{equation}
\begin{equation}\label{S_{2, 2}}
|S_{2, 2}(\tilde{\textbf{x}})|\leq C\frac{r+1}{r}\left(\Big(\frac{r}{r+1}\delta\Big)^{1/2}+\tilde{\rho}\right)^{-1},
\end{equation}
and
\begin{equation}\label{widetilde{S}_{2, 1}}
|\widetilde{S}_{2, 1}(\tilde{\textbf{x}})-\mbox{T}_{2}|\leq C\frac{r+1}{r}\left(\Big(\frac{r}{r+1}\delta\Big)^{1/2}+\tilde{\rho}\right)^{-1}.
\end{equation}

Then, from \eqref{S_{2, 1}}--\eqref{widetilde{S}_{2, 1}},
\begin{equation*}
\partial_{\tilde{x}_{1}}v_{1}=\left(\frac{r}{r+1}+\frac{r+1}{r}O(|\log \delta|^{-1})\right)\partial_{\tilde{x}_{1}}v_{1}^{0}+R,
\end{equation*}
where
$$|R|\leq C\left(\frac{r+1}{r}\right)^{2}\left(|\log \delta|^{2}+\left(\Big(\frac{r}{r+1}\delta\Big)^{1/2}+\tilde{\rho}\right)^{-1}\right).$$
Hence,
\begin{equation*}
\partial_{\tilde{x}_{1}}v_{1}=\left(\frac{r}{r+1}+\frac{r+1}{r}O(|\log \delta|^{-1})\right)\partial_{\tilde{x}_{1}}v_{1}^{0}.
\end{equation*}
Similarly, we have
\begin{equation*}
\partial_{\tilde{x}_{1}}v_{2}=-\left(\frac{1}{r+1}+\left(\frac{r+1}{r}\right)^{2}O(|\log \delta|^{-1})\right)\partial_{\tilde{x}_{1}}v_{2}^{0}.
\end{equation*}
So Lemma \ref{lem x1 v} is an immediate consequence.

The rest is to prove \eqref{S_{2, 1}}--\eqref{widetilde{S}_{2, 1}}.

\noindent{\bf Step 2. Proof of \eqref{S_{2, 1}}.}
We obtain from \eqref{y_{1,2k}} that
$$\mathscr{P}_{1, 2k}\leq \mathscr{P}_{1, 2N_{0}(\delta)}\leq C\frac{r}{r+1}\frac{1}{|\log \delta|}\quad \text{for}~k\geq N_{0}(\delta).$$
This means that, for $k\geq N_{0}(\delta)$, we have
$$\mathscr{P}_{1, 2k}=\frac{r}{r+1}O(|\log \delta|^{-1}).$$
Hence, for $k\geq N_{0}(\delta)$, by using \eqref{p_{1,2k}p_{1,2k+1}}, we have
\begin{align}\label{y_{1, 2k+1}-y_{1, 2k+3}}
\mathscr{P}_{1, 2k+3}-\mathscr{P}_{1, 2k+1}
&=\mu_{1, 2k+1}\mu_{1, 2k+3}\Big(\mathscr{P}_{1, 2k}-\mathscr{P}_{1, 2k+2}\Big)\nonumber\\
&=(1+O(|\log \delta|^{-1}))\Big(\mathscr{P}_{1, 2k}-\mathscr{P}_{1, 2k+2}\Big),
\end{align}
and
\begin{align}\label{y_{1, 2k+1}^{2}-(frac{p_{2}}{r_{1}})^{2}}
\mathscr{P}_{1, 2k+1}^{2}-\mathscr{P}_{2}^{2}&=\mathscr{P}_{1, 2k}^{2}+\left(\frac{r}{r+1}\right)^{2}O(|\log \delta|^{-3})-\mathscr{P}_{1}^{2}\nonumber\\
&=\left(1+O(|\log \delta|^{-1})\right)\left(\mathscr{P}_{1, 2k}^{2}-\mathscr{P}_{1}^{2}\right).
\end{align}
It follows from \eqref{y_{1, 2k+1}-y_{1, 2k+3}}, \eqref{y_{1, 2k+1}^{2}-(frac{p_{2}}{r_{1}})^{2}}, and Lemma \ref{lem finer p q} \eqref{lem3.3-1} that
\begin{equation*}
S_{2, 1}(\tilde{\textbf{x}})=\left(\frac{r}{r+1}+\frac{r+1}{r}O(|\log \delta|^{-1})\right)\widetilde{S}_{2, 1}(\tilde{\textbf{x}}).
\end{equation*}
\eqref{S_{2, 1}} is thus proved.

\noindent{\bf Step 3. Proof of \eqref{S_{2, 2}}.}
We set
\begin{align}\label{def S22}
S_{2, 2}(\tilde{\textbf{x}})
&=\sum\limits_{k=N_{0}(\delta)}^{N_{1}(\delta)-1}q_{1, 2k}(\tilde{x}_{1}-\mathscr{P}_{1, 2k+1})(\mu_{1, 2k+1}-1)t_{2}^{-3/2}\nonumber\\
&\quad+\sum\limits_{k=N_{0}(\delta)}^{N_{1}(\delta)-1}q_{1, 2k}\left((\tilde{x}_{1}-\mathscr{P}_{1, 2k+1})t_{2}^{-3/2}-(\tilde{x}_{1}+\mathscr{P}_{1, 2k})t_{1}^{-3/2}\right)\nonumber\\
&=:\eta_{1}+\eta_{2},
\end{align}
where
\begin{align*}
|\eta_{1}|&\leq\sum\limits_{k=N_{0}(\delta)}^{N_{1}(\delta)-1}q_{1, 2k}|\mu_{1, 2k+1}-1|t_{2}^{-1}=\sum\limits_{k=N_{0}(\delta)}^{N(\delta)-1}+\sum\limits_{k=N(\delta)}^{N_{1}(\delta)-1}q_{1, 2k}|\mu_{1, 2k+1}-1|t_{2}^{-1}\\
&=:\eta_{1, 1}+\eta_{1, 2}.
\end{align*}

By using \eqref{q_{1,2k}leq}, \eqref{|1-mu_{1, 2k}|}, and \eqref{(tilde{x}-y_{1, 2k+1})^{2}}, we have
\begin{equation*}
\eta_{1, 1}\leq\sum\limits_{k=N_{0}(\delta)}^{N(\delta)-1}\left(\frac{r+1}{r}\right)^{2}
\leq C\frac{r+1}{r}\Big(\frac{r\delta}{r+1}\Big)^{-1/2}\quad\mbox{if}~\tilde{\rho}< \Big(\frac{r\delta}{r+1}\Big)^{1/2},
\end{equation*}
and
\begin{align*}
\eta_{1, 1}&\leq C\sum\limits_{k=N_{0}(\delta)+1}^{N(\delta)}\left(\tilde{\rho}^{2}(k+1)^{2}+\left(\frac{r}{r+1}\right)^{2}\right)^{-1}\\
&\leq C\left(\frac{r+1}{r}\right)^{2}\int_{N_{0}(\delta)}^{N(\delta)}\left(1+\left(\frac{r+1}{r}\tilde{\rho}s\right)^{2}\right)^{-1}ds\\
&\leq C\frac{r+1}{r}\frac{1}{\tilde{\rho}}\quad\mbox{if}~\tilde{\rho}>\Big(\frac{r\delta}{r+1}\Big)^{1/2}.
\end{align*}
It follows from \eqref{2.22}, \eqref{|1-mu_{1, 2k}|'}, and \eqref{(tilde{x}-y_{1, 2k+1})'^{2}} that
\begin{align*}
\eta_{1, 2}&\leq \left(\frac{r+1}{r}\right)^{1/2}\sum\limits_{k=N(\delta)}^{N_{1}(\delta)-1}q_{1, 2k}\sqrt{\delta}\left(\frac{r\delta}{r+1}+\tilde{\rho}^{2}\right)^{-1}\\
&\leq C\frac{r+1}{r}\left(\tilde{\rho}+\Big(\frac{r\delta}{r+1}\Big)^{1/2}\right)^{-1}.
\end{align*}
This means that
\begin{equation}\label{eta_{1}}
|\eta_{1}|\leq C\frac{r+1}{r}\left(\tilde{\rho}+\Big(\frac{r\delta}{r+1}\Big)^{1/2}\right)^{-1}.
\end{equation}

Notice that
\begin{align*}
|\eta_{2}|
&\leq\sum\limits_{k=N_{0}(\delta)}^{N(\delta)-1}+\sum\limits_{k=N(\delta)}^{N_{1}(\delta)-1}q_{1, 2k}|\mathscr{P}_{1, 2k}+\mathscr{P}_{1, 2k+1}|t_{1}^{-3/2}\\
&\quad+\sum\limits_{k=N_{0}(\delta)}^{N(\delta)-1}+\sum\limits_{k=N(\delta)}^{N_{1}(\delta)-1}q_{1, 2k}|\tilde{x}_{1}-\mathscr{P}_{1, 2k+1}|\left|t_{2}^{-3/2}-t_{1}^{-3/2}\right|\\
&=:\eta_{2, 1}+\eta_{2, 2}+\eta_{2, 3}+\eta_{2, 4}.
\end{align*}
We obtain from \eqref{y_{1, 2k}+y_{1, 2k+1}}, \eqref{q_{1,2k}leq}, and \eqref{(tilde{x}+y_{1, 2k})^{2}} that if $\tilde{\rho}< \Big(\frac{r}{r+1}\delta\Big)^{1/2}$, then
\begin{equation}\label{eta 21'}
\eta_{2, 1}\leq C\sum\limits_{k=N_{0}(\delta)}^{N(\delta)-1}\left(\frac{r+1}{r}\right)^{2}\left(1+\frac{1}{k}\right)^{2}
\leq C\frac{r+1}{r}\Big(\frac{r\delta}{r+1}\Big)^{-1/2}.
\end{equation}
If $\tilde{\rho}>\Big(\frac{r}{r+1}\delta\Big)^{1/2}$, then
\begin{align}\label{eta 21}
\eta_{2, 1}&\leq\sum\limits_{k=N_{0}(\delta)}^{N(\delta)-1}\frac{r}{r+1}\frac{1}{k^{2}}\frac{1}{k+1}\left(\tilde{\rho}^{2}+(\Theta_{1,2k+1}(r))^{2}\right)^{-\frac{3}{2}}\nonumber\\
&\leq C\frac{r}{r+1}\int_{N_{0}(\delta)}^{N(\delta)}
\left(s^{2}\tilde{\rho}^{2}+\left(\frac{r}{r+1}\right)^{2}\right)^{-\frac{3}{2}}\ ds
\leq C\frac{r+1}{r}\frac{1}{\tilde{\rho}}.
\end{align}
Combining \eqref{2.22}, \eqref{(tilde{x}+y_{1, 2k})'^{2}}, and \eqref{|y_{1, 2k+1}+y_{1, 2k}|'}, we deduce
\begin{equation}\label{eta 22}
\eta_{2, 2}\leq C\sum\limits_{k=N(\delta)}^{N_{1}(\delta)-1}q_{1, 2k}\delta\left(\tilde{\rho}^{2}+\frac{r\delta}{r+1}\right)^{-3/2}
\leq C\frac{r+1}{r}\left(\Big(\frac{r\delta}{r+1}\Big)^{1/2}+\tilde{\rho}\right)^{-1}.
\end{equation}
By using \eqref{3.12} and the same argument that led to \eqref{eta 21} and \eqref{eta 22}, we get $\eta_{2, 3}$ and $\eta_{2, 4}$ bounded by the right-hand side of \eqref{eta 22}. Thus,
\begin{equation}\label{eta_{2}}
|\eta_{2}|\leq C\frac{r+1}{r}\left(\Big(\frac{r\delta}{r+1}\Big)^{1/2}+\tilde{\rho}\right)^{-1}.
\end{equation}
Coming back to \eqref{def S22}, we get \eqref{S_{2, 2}} by using \eqref{eta_{1}} and \eqref{eta_{2}}.

\noindent{\bf Step 4. Proof of \eqref{widetilde{S}_{2, 1}}.}
For $N_{0}(\delta)\leq k\leq N_{1}(\delta)$, let
\begin{equation*}
\gamma_{2k}=\Big|\partial_{\tilde{x}_{1}}f_{1}(\mathscr{P}_{1,2k})(\mathscr{P}_{1, 2k}-\mathscr{P}_{1, 2k+2})+\int_{\mathscr{P}_{1, 2k}}^{\mathscr{P}_{1, 2k+2}}\partial_{\tilde{x}_{1}}f_{1}(t)\ dt\Big|.
\end{equation*}
Define
\begin{equation*}
f(t):=\partial_{\tilde{x}_{1}}f_{1}(t).
\end{equation*}
By the mean value property, there is $t_{1, 2k}\in [\mathscr{P}_{1, 2k+2}, \mathscr{P}_{1, 2k}]$ such that
$$f(\mathscr{P}_{1, 2k})(\mathscr{P}_{1, 2k}-\mathscr{P}_{1, 2k+2})+\int_{\mathscr{P}_{1, 2k}}^{\mathscr{P}_{1, 2k+2}}f(t)\ dt=\frac{1}{2}f'(t_{1, 2k})(\mathscr{P}_{1, 2k}-\mathscr{P}_{1, 2k+2})^{2}.$$
Then, we have
\begin{align*}
\gamma_{2k}
&\leq\frac{1}{2}\left|\partial_{t}\partial_{\tilde{x}_{1}}\left(|\tilde{\textbf{x}}-(t, 0, 0)|^{-1}-|\tilde{\textbf{x}}+(t, 0, 0)|^{-1}\right)\big|_{t=t_{1, 2k}}\right|
\frac{(\mathscr{P}_{1, 2k}-\mathscr{P}_{1, 2k+2})^{2}}{\sqrt{t_{1, 2k}^{2}-\mathscr{P}_{1}^{2}}}\\
&\quad+\frac{1}{2}|f(t_{1,2k})t_{1,2k}|\frac{(\mathscr{P}_{1, 2k}-\mathscr{P}_{1, 2k+2})^{2}}{t_{1, 2k}^{2}-\mathscr{P}_{1}^{2}}\\
&=:\frac{1}{2}(\gamma_{2k, 1}+\gamma_{2k, 2}).
\end{align*}

First, recalling \eqref{|tilde{x}|}, one can show that there is some constant $C$ independent of $k$, such that
$$|\tilde{\textbf{x}}\pm(t_{1, 2k}, 0, 0)|^{2}\geq C(t_{1, 2k}^{2}+\tilde{\rho}^{2}) \quad  \text{for}~x_{1}\in R_{\delta}.$$
We thus have
\begin{equation*}
\gamma_{2k, 1}\leq \frac{C}{\tilde{\rho}^{3}+|t_{1, 2k}|^{3}}
\frac{1}{\sqrt{t_{1, 2k}^{2}-\mathscr{P}_{1}^{2}}}(\mathscr{P}_{1, 2k}-\mathscr{P}_{1, 2k+2})^{2}
\end{equation*}
and
\begin{equation*}
\gamma_{2k, 2}\leq \frac{C}{\tilde{\rho}^{2}+|t_{1, 2k}|^{2}}
|t_{1, 2k}|\left(t_{1, 2k}^{2}-\mathscr{P}_{1}^{2}\right)^{-\frac{3}{2}}(\mathscr{P}_{1, 2k}-\mathscr{P}_{1, 2k+2})^{2}.
\end{equation*}

If $k\leq N(\delta)-1$, then we have
\begin{equation*}
|t_{1, 2k}|\simeq\frac{r}{r+1}\frac{1}{k+1}.
\end{equation*}
In view of \eqref{y_{1, 2k}+y_{1, 2k+1}}, one can see that
\begin{equation*}
|\mathscr{P}_{1, 2k}-\mathscr{P}_{1, 2k+2}|\leq C\frac{r}{r+1}\frac{1}{k^{2}}.
\end{equation*}
By using Lemma \ref{lem3.2} \eqref{lem3.2-3},
\begin{equation*}
\mathscr{P}_{1, 2k+2}-\mathscr{P}_{1}\geq C\frac{r}{r+1}\frac{1}{k+1}.
\end{equation*}
Therefore, we have
\begin{equation*}
\sum_{k=N_{0}(\delta)}^{N(\delta)-1}\gamma_{2k, 1}\leq C\frac{r}{r+1}\sum_{k=N_{0}(\delta)}^{N(\delta)-1}\frac{1}{\tilde{\rho}^{3}(k+1)^{3}+\left(\frac{r}{r+1}\right)^{3}}.
\end{equation*}
By using the same argument that led to \eqref{eta 21'} and \eqref{eta 21}, we have
\begin{equation*}
\sum_{k=N_{0}(\delta)}^{N(\delta)-1}\gamma_{2k, 1}\leq C\left(1+\frac{1}{r}\right)^{3/2}\delta^{-1/2}\quad\mbox{if}~\tilde{\rho}\leq \Big(\frac{r\delta}{r+1}\Big)^{1/2},
\end{equation*}
and
\begin{equation*}
\sum_{k=N_{0}(\delta)}^{N(\delta)-1}\gamma_{2k, 1}
\leq C\frac{r+1}{r}\frac{1}{\tilde{\rho}}\quad\mbox{if}~\tilde{\rho}> \Big(\frac{r\delta}{r+1}\Big)^{1/2}.
\end{equation*}

If $N(\delta)\leq k\leq N_{1}(\delta)-1$, we have from \eqref{2.4} that
\begin{align*}
\mathscr{P}_{1, 2k}-\mathscr{P}_{1, 2k+2}&=\frac{G_{1}A_{1}^{k+1}(1-A_{1}^{-1})}{(G_{1}A_{1}^{k+1}-B_{1})(G_{1}A_{1}^{k}-B_{1})}\\
&\leq C \frac{A_{1}^{k+1}(1-A_{1}^{-1})}{B_{1}(A_{1}^{k+1}-1)(A_{1}^{k}-1)},
\end{align*}
the last inequality holds since $\varepsilon$ is sufficiently small and
$$G_{1}=(1+\delta-\mathscr{P}_{1})^{-1}+B_{1}>B_{1}.$$
A direct calculation gives that
\begin{equation*}
\mathscr{P}_{1, 2k}-\mathscr{P}_{1, 2k+2}\leq C\frac{A_{1}^{-k}(1-A_{1}^{-1})}{B_{1}(1-A_{1}^{-(k+1)})(1-A_{1}^{-k})}
\leq C\delta A_{1}^{-k}.
\end{equation*}
By using Lemma \ref{lem3.2} \eqref{lem3.2-2}, for all $k\geq 0$, we have
\begin{equation*}
\mathscr{P}_{1, 2k}-\mathscr{P}_{1}
\geq 2\Big(\frac{r\delta}{r+1}\Big)^{1/2}A^{-k}_{1}.
\end{equation*}
Since $\mathscr{P}_{1, 2k}\geq\mathscr{P}_{1}$ for all $k\geq 0$, we have
\begin{align*}
\sum_{k=N(\delta)}^{N_{1}(\delta)-1}\gamma_{2k, 1}&\leq \sum_{k=N(\delta)}^{N_{1}(\delta)-1} \frac{C}{\tilde{\rho}^{3}+\left(\frac{r}{r+1}\delta\right)^{3/2}}
\frac{1}{\sqrt{\frac{r}{r+1}\delta A^{-k}_{1}}}\delta^{2}A^{-2k}_{1}\\
&\leq C\left(\frac{r+1}{r}\right)^\frac{1}{2}\delta^{3/2}\left(\tilde{\rho}^{3}+\left(\frac{r\delta}{r+1}\right)^{3/2}\right)^{-1}
\sum_{k=N(\delta)}^{\infty}A^{-\frac{3}{2}k}_{1}\\
&\leq C\delta\left(\tilde{\rho}^{3}+\left(\frac{r\delta}{r+1}\right)^{3/2}\right)^{-1}
\leq C\frac{r+1}{r}\left(\Big(\frac{r\delta}{r+1}\Big)^{1/2}+\tilde{\rho}\right)^{-1}.
\end{align*}
Thus, we have
\begin{equation}\label{gamma 2k 1}
\sum_{k=N_{0}(\delta)}^{N_{1}(\delta)-1}\gamma_{2k, 1}\leq C\frac{r+1}{r}\left(\Big(\frac{r\delta}{r+1}\Big)^{1/2}+\tilde{\rho}\right)^{-1}.
\end{equation}
Similarly, one can show that $\sum_{k=N_{0}(\delta)}^{N_{1}(\delta)-1}\gamma_{2k, 2}$ is also bounded by the right-hand side of \eqref{gamma 2k 1}. This completes the proof of \eqref{widetilde{S}_{2, 1}}.

\section{Proof of Proposition \ref{prop u}}\label{sec pf prop u}

It has been proved in \cite{Y1} and \cite{Y2} that
\begin{equation}\label{u_{1}-u_{2}}
u|_{\partial \mathfrak{B}_{1}}-u|_{\partial \mathfrak{B}_{2}}=\int_{\partial \mathfrak{B}_{1}}H\frac{\partial h}{\partial \nu^{1}}\ d\sigma+\int_{\partial \mathfrak{B}_{2}}H\frac{\partial h}{\partial \nu^{2}}\ d\sigma.
\end{equation}
Since $h$ is a constant on $\partial \mathfrak{B}_{1}$ and $\partial \mathfrak{B}_{2}$, one can see from \eqref{u_{1}-u_{2}}, \eqref{def h}, and Green's representation formula that
\begin{align}\label{u boundary}
&u|_{\partial \mathfrak{B}_{1}}-u|_{\partial \mathfrak{B}_{2}}\nonumber\\
&=\frac{Q_{2}}{M}\sum_{j=0}^{\infty}(-1)^{j}q_{1, j}\int_{\partial (\mathfrak{B}_{1}\cup \mathfrak{B}_{2})}\Gamma(\textbf{x}-\textbf{p}_{1, j})\frac{\partial H}{\partial \nu}-H\frac{\partial \Gamma}{\partial \nu}(\textbf{x}-\textbf{p}_{1, j})~d\sigma\nonumber\\
&\quad+\frac{Q_{1}}{M}\sum_{j=0}^{\infty}(-1)^{j+1}q_{2, j}\int_{\partial (\mathfrak{B}_{1}\cup \mathfrak{B}_{2})}\Gamma(\textbf{x}-\textbf{p}_{2, j})\frac{\partial H}{\partial \nu}-H\frac{\partial \Gamma}{\partial \nu}(\textbf{x}-\textbf{p}_{2, j})~d\sigma\nonumber\\
&=\frac{Q_{2}}{M}\sum_{j=0}^{\infty}(-1)^{j}q_{1, j}H(\textbf{p}_{1, j})+\frac{Q_{1}}{M}\sum_{j=0}^{\infty}(-1)^{j+1}q_{2, j}H(\textbf{p}_{2, j}).
\end{align}

Without loss of generality, we may assume that $H(0,0,0)=0$. Then for $k\leq N(\delta)$, by Lemma \ref{lem3.1}, we have
\begin{align*}
&\left|q_{1,2k+1}H(p_{1,2k+1})-\Theta_{1,2k+1}(r)H\left(-\Theta_{1,2k+1}(r)r_{1}, 0, 0\right)\right|\\
&\leq\left|q_{1,2k+1}-\Theta_{1,2k+1}(r)\right||H(\textbf{p}_{1,2k+1})-H(\textbf{0})|\\
&\quad+\Theta_{1,2k+1}(r)\left|H(\textbf{p}_{1,2k+1})-H\left(-\Theta_{1,2k+1}(r)r_{1}, 0, 0\right)\right|\\
&\leq C\Big(\frac{r+1}{r}\delta\Big)^{1/2}\left(\Theta_{1,2k+1}(r)+\sqrt{\delta}\right)r_{1}
+C\sqrt{\delta}\Theta_{1,2k+1}(r)r_{1},
\end{align*}
and
\begin{align*}
&\left|q_{1,2k}H(p_{1,2k})-\Theta_{1,2k}(r)H\left(\Theta_{1,2k}(r)r_{1}, 0, 0\right)\right|\\
&\leq C\Big(\frac{r+1}{r}\delta\Big)^{1/2}\left(\Theta_{1,2k}(r)+\sqrt{\delta}\right)r_{1}
+C\sqrt{\delta}\Theta_{1,2k}(r)r_{1}.
\end{align*}
We thus have
\begin{align*}
&\left|\sum_{k=0}^{N(\delta)-1}\left(q_{1,2k}H(p_{1,2k})-q_{1,2k+1}H(p_{1,2k+1})\right)\right.\\
&\left.-\sum_{k=0}^{N(\delta)-1}\left(\Theta_{1,2k}(r)H\Big(\Theta_{1,2k}(r)r_{1}, 0, 0, 0\Big)-\Theta_{1,2k+1}(r)H\Big(-\Theta_{1,2k+1}(r)r_{1}, 0, 0\Big)\right)\right|\\
&\leq C\sqrt{\delta}|\log\delta|r_{1},
\end{align*}
where $C>0$ is independent of $\varepsilon$.

On the other hand, by Lemma \ref{lem3.1}, Lemma \ref{lem3.2} \eqref{lem3.2-1} and $H(\textbf{0})=0$, we have
\begin{equation*}
\left|\sum_{k=N(\delta)}^{{\infty}}q_{1,2k+1}H(\textbf{p}_{1,2k+1})\right|
\leq C\Big(\frac{r\delta}{r+1}\Big)^{1/2}r_{1},
\end{equation*}
and
\begin{equation*}
\left|\sum_{k=N(\delta)}^{{\infty}}q_{1,2k}H(\textbf{p}_{1,2k})\right|
\leq C\Big(\frac{r\delta}{r+1}\Big)^{1/2}r_{1}.
\end{equation*}
Moreover,
\begin{equation*}
\left|\sum_{k=N(\delta)}^{{\infty}}\Theta_{1,2k+1}(r)H\left(-\Theta_{1,2k+1}(r)r_{1}, 0, 0\right)\right|\leq C\sqrt{\delta}r_{1},
\end{equation*}
and similarly,
\begin{equation*}
\left|\sum_{k=N(\delta)}^{{\infty}}\Theta_{1,2k}(r)H\left(\Theta_{1,2k}(r)r_{1}, 0, 0\right)\right|\leq C\sqrt{\delta}r_{1}.
\end{equation*}
Combining the above estimates, we obtain
\begin{align*}
&\Bigg|\sum_{j=0}^{\infty}(-1)^{j}q_{1, j}H(\textbf{\mbox{p}}_{1, j})-\sum_{k=0}^{\infty}\Theta_{1,2k}(r)H\left(\Theta_{1,2k}(r)r_{1}, 0, 0\right)\\
&\quad+\sum_{k=0}^{\infty}\Theta_{1,2k+1}(r)H\left(-\Theta_{1,2k+1}(r)r_{1}, 0, 0\right)\Bigg|\leq C\sqrt{\delta}|\log\delta|r_{1}.
\end{align*}
By the similar way, we have
\begin{align*}
&\Bigg|\sum_{j=0}^{\infty}(-1)^{j+1}q_{2, j}H(\textbf{\mbox{p}}_{2, j})-\sum_{k=0}^{\infty}\Theta_{2,2k+1}(r)H\left(\Theta_{2,2k+1}(r)r_{2}, 0, 0\right)\\
&\quad+\sum_{k=0}^{\infty}\Theta_{2,2k}(r)H\left(-\Theta_{2,2k}(r)r_{2}, 0, 0\right)\leq C\sqrt{\delta}|\log\delta|r_{2},
\end{align*}
where
\begin{equation*}
\Theta_{2,2k}(r)=\frac{1}{k(1+r)+1},\quad\Theta_{2,2k+1}(r)=\frac{1}{(k+1)(1+r)}.
\end{equation*}
Therefore, coming back to \eqref{u boundary} and recalling the definitions of $C_{\min}^{H}$ and $C_{\max}^{H}$, \eqref{CH min} and \eqref{CH max}, we obtain from Proposition \ref{prop Q} that
\begin{equation}\label{u bound diff}
u|_{\partial \mathfrak{B}_{1}}-u|_{\partial \mathfrak{B}_{2}}=\frac{2\Psi(r_{1}, r_{2})+O\big(\sqrt{\varepsilon}\big|\log\varepsilon\big|\big)}{|\log\varepsilon|}\left(1+\frac{r_{1}+r_{2}}{r_{2}}O(|\log\varepsilon|^{-1})\right),
\end{equation}
where $\Psi(r_{1}, r_{2})$ is defined by \eqref{def psi}. Proposition \ref{prop u} is thus proved.

\begin{remark}
When $H(\textbf{x})=E_{0}(\textbf{x}\cdot\textbf{n})=E_{0}x_{1}$ for some constant $E_{0}$, combining \eqref{u bound diff} and \eqref{def psi}, we can get
\begin{align}\label{difference 1}
&\frac{u|_{\partial \mathfrak{B}_{1}}-u|_{\partial \mathfrak{B}_{2}}}{\varepsilon}\nonumber\\
&\approx\frac{\pi^{2}E_{0}}{3\varepsilon|\log\varepsilon|}\frac{r_{1}r_{2}}{r_{1}+r_{2}}\left(1+\frac{6}{\pi^{2}}\frac{\psi\Big(\frac{r_{1}}{r_{1}+r_{2}}\Big)\psi'\Big(\frac{r_{2}}{r_{1}+r_{2}}\Big)
+\psi\Big(\frac{r_{2}}{r_{1}+r_{2}}\Big)\psi'\Big(\frac{r_{1}}{r_{1}+r_{2}}\Big)}
{\psi\Big(\frac{r_{1}}{r_{1}+r_{2}}\Big)+\psi\Big(\frac{r_{2}}{r_{1}+r_{2}}\Big)}\right).
\end{align}
This is much more concise than the formula (57) in \cite{L}. In fact, the author \cite{L} proved that the average field which is the potential difference divided by the distance $\varepsilon$ between two spheres with different radii, is given by
\begin{align}\label{difference 2}
&\frac{u|_{\partial \mathfrak{B}_{1}}-u|_{\partial \mathfrak{B}_{2}}}{\varepsilon}\nonumber\\
&\approx\frac{\pi^{2}E_{0}}{6\varepsilon}\frac{r_{1}r_{2}}{r_{1}+r_{2}}\Bigg(\psi\Big(\frac{r_{1}}{r_{1}+r_{2}}\Big)+\psi\Big(\frac{r_{2}}{r_{1}+r_{2}}\Big)+\frac{6}{\pi^{2}}\Big(\psi\Big(\frac{r_{1}}{r_{1}+r_{2}}\Big)\psi'\Big(\frac{r_{2}}{r_{1}+r_{2}}\Big)\nonumber\\
&\quad+\psi\Big(\frac{r_{2}}{r_{1}+r_{2}}\Big)\psi'\Big(\frac{r_{1}}{r_{1}+r_{2}}\Big)\Big)\Bigg)\Bigg(\left(\frac{1}{2}\log\frac{2r_{1}r_{2}}{(r_{1}+r_{2})\varepsilon}+\gamma\right)\nonumber\\
&\quad\cdot\left(\psi\Big(\frac{r_{1}}{r_{1}+r_{2}}\Big)+\psi\Big(\frac{r_{2}}{r_{1}+r_{2}}\Big)\right)-\psi\Big(\frac{r_{1}}{r_{1}+r_{2}}\Big)\psi\Big(\frac{r_{2}}{r_{1}+r_{2}}\Big)\Bigg)^{-1},
\end{align}
where $\gamma$ is Euler-Mascheroni constant. For given $r_{1},r_{2}\gg\varepsilon$, formulae \eqref{difference 1} and \eqref{difference 2} coincide up to $|\varepsilon\log\varepsilon|^{-1}$ for small enough $\varepsilon>0$.
\end{remark}

\section{Appendix}\label{append}

For the completeness of this paper, we give the proof of Lemma \ref{lem finer p q} by adapting the idea in \cite{KLY2}.

\begin{proof}[\bf Proof of Lemma \ref{lem finer p q}.]
\noindent{\bf STEP 1. Proof of \eqref{-y_{1, 2k}+y_{1, 2k+2}}.}

\noindent{\bf STEP 1.1.}
If $k>l$, then it follows from \eqref{q_{1}}, \eqref{mu 1 2k}, and \eqref{mu 1 2k+1} that
\begin{equation*}
\log q_{1, 2k}=-\sum_{m=l}^{k-1}\left(\log (1+\delta-\mathscr{P}_{1, 2m+1})+\log \left(\frac{r+\delta+\mathscr{P}_{1, 2m}}{r}\right)
\right)+\log q_{1, 2l}.
\end{equation*}
By using the inequality $|\log(1+t)-t|\leq Ct^{2}$, we obtain
\begin{equation}\label{eq log q12k}
\log q_{1, 2k}=-\sum_{m=l}^{k-1}\left(\frac{\delta+\mathscr{P}_{1, 2m}}{r}+\delta-\mathscr{P}_{1, 2m+1}\right)+\log q_{1, 2l}+E_{1},
\end{equation}
where the error term $E_{1}$ satisfies
\begin{align*}
|E_{1}|&\leq C\sum_{m=l}^{k-1}\left(r^{-2}\Big(\delta+\mathscr{P}_{1, 2m}\Big)^{2}+\Big(\delta-\mathscr{P}_{1, 2m+1}\Big)^{2}\right)\\
&\leq C\sum_{m=l}^{k-1}\left(r^{-2}\mathscr{P}_{1, 2m}^{2}+\mathscr{P}_{1, 2m+1}^{2}\right).
\end{align*}
The last inequality above holds since $\varepsilon$ is sufficiently small. We have from \eqref{eq log q12k}, \eqref{2.4}, and \eqref{2.5} that
\begin{equation*}
\log \frac{q_{1, 2k}}{q_{1, 2l}}=-(k-l)\left(\frac{r+1}{r}\delta+\frac{1}{r}\mathscr{P}_{1}-\mathscr{P}_{2}\right)-\frac{1}{r}\sum_{m=l}^{k-1}f_{1}(m)
+\sum_{m=l}^{k-1}f_{2}(m)+E_{1},
\end{equation*}
where
\begin{align*}
&f_{i}(m)=\frac{A_{i}^{-m}}{G_{i}-B_{i}A_{i}^{-m}},~i=1,2,\\
&G_{1}=(1+\delta-\mathscr{P}_{1})^{-1}+B_{1},\quad
G_{2}=\Big(-\frac{r}{r+1}+O(\delta)-\mathscr{P}_{2}\Big)^{-1}+B_{2}.
\end{align*}
Since $f_{1}(m)$ is decreasing in $m$ and $f_{2}(m)$ is increasing in $m$, we have
\begin{equation*}
\left|\sum_{m=l}^{k-1}f_{1}(m)
+\frac{1}{B_{1}\log A_{1}}\log f_{1}^{k,l}\right|
=\left|\sum_{m=l}^{k-1}f_{1}(m)
-\int_{l}^{k}f_{1}(x)\ dx\right|\leq f_{1}(l),
\end{equation*}
and
\begin{equation*}
\left|\sum_{m=l}^{k-1}f_{2}(m)+\frac{1}{B_{2}\log A_{2}}\log f_{2}^{k,l}\right|
\leq -f_{2}(l),
\end{equation*}
where
\begin{equation}\label{F_{j}}
f_{i}^{k,l}=\frac{G_{i}-B_{i}A_{i}^{-l}}
{G_{i}-B_{i}A_{i}^{-k}},
\quad i=1, 2.
\end{equation}
Thus, we have
\begin{align*}
\log \frac{q_{1, 2k}}{q_{1, 2l}}&=-(k-l)\left(\frac{r+1}{r}\delta+\frac{1}{r}\mathscr{P}_{1}-\mathscr{P}_{2}\right)\\
&\quad+\frac{1}{r}\frac{1}{B_{1}\log A_{1}}\log f_{1}^{k,l}-\frac{1}{B_{2}\log A_{2}}\log f_{2}^{k,l}+E_{2},
\end{align*}
where the new error term $E_{2}$ satisfies
\begin{equation}\label{E_{2}}
|E_{2}|\leq C\sum_{m=l}^{k-1}\left(\frac{1}{r^{2}}\mathscr{P}_{1, 2m}^{2}+\mathscr{P}_{1, 2m+1}^{2}\right)-f_{2}(l)+\frac{1}{r}f_{1}(l).
\end{equation}
One can see from \eqref{A_{1}} and \eqref{B_{1}} that
\begin{equation*}
\frac{1}{B_{1}\log A_{1}}=\frac{r}{r+1}+E_{3,1}, \qquad \frac{1}{B_{2}\log A_{2}}=-\frac{r}{r+1}+E_{3,2},
\end{equation*}
where
\begin{equation}\label{E_{3}}
|E_{3,1}|\leq C \sqrt{\delta}, \quad |E_{3,2}|\leq C \sqrt{\delta}.
\end{equation}
Then, we have
\begin{align*}
\log \frac{q_{1, 2k}}{q_{1, 2l}}&=-(k-l)\left(\frac{r+1}{r}\delta+\frac{1}{r}\mathscr{P}_{1}-\mathscr{P}_{2}\right)\\
&\quad+\frac{1}{r}\Big(\frac{r}{r+1}+E_{3,1}\Big)\log f_{1}^{k,l}+\Big(\frac{r}{r+1}+E_{3,2}\Big)\log f_{2}^{k,l}+E_{2},
\end{align*}
which in turn implies
\begin{equation}\label{q_{1, 2k}q_{1, 2m}}
q_{1, 2k}=q_{1, 2l}\exp\left(k\Big(\mathscr{P}_{2}-\frac{1}{r}\mathscr{P}_{1}\Big)\right){(f_{1}^{k,l})}^{\frac{1}{r+1}}{(f_{2}^{k,l})}^{\frac{r}{r+1}}\exp({E_{4}}),
\end{equation}
where
\begin{equation}\label{E_{4}}
E_{4}:=-\frac{r+1}{r}(k-l)\delta-l\left(\mathscr{P}_{2}-\frac{1}{r}\mathscr{P}_{1}\right)
+E_{2}+\frac{1}{r}E_{3,1}\log f_{1}^{k,l}+E_{3,2}\log f_{2}^{k,l}.
\end{equation}

By \eqref{2.4} and \eqref{2.5}, we have
\begin{equation*}
\mathscr{P}_{1, 2k}-\mathscr{P}_{1, 2k+2}=\left(\mathscr{P}_{1, 2k+2}-\mathscr{P}_{1}\right)\frac{(A_{1}-1)G_{1}}{G_{1}-B_{1}A_{1}^{-k}},
\end{equation*}
\begin{equation*}
\mathscr{P}_{1, 2k+3}-\mathscr{P}_{1, 2k+1}=\left(\mathscr{P}_{2}-\mathscr{P}_{1, 2k+3}\right)\frac{(A_{2}-1)G_{2}}{G_{2}-B_{2}A_{2}^{-k}}.
\end{equation*}
Therefore, we obtain
\begin{align}\label{q_{1, 2k}'}
&q_{1, 2k}\Big(\mathscr{P}_{1, 2k}-\mathscr{P}_{1, 2k+2}\Big)^{-\frac{1}{r+1}}\Big(\mathscr{P}_{1, 2k+3}-\mathscr{P}_{1, 2k+1}\Big)^{-\frac{r}{r+1}}\nonumber\\
&=q_{1, 2l}\exp\left(k\Big(\mathscr{P}_{2}-\frac{1}{r}\mathscr{P}_{1}\Big)\right)\left(\mathscr{P}_{1, 2k+2}-\mathscr{P}_{1}\right)^{-\frac{1}{r+1}}\nonumber\\
&\quad\cdot \left(\mathscr{P}_{2}-\mathscr{P}_{1, 2k+3}\right)^{-\frac{r}{r+1}}H_{1}^{\frac{1}{r+1}}H_{2}^{\frac{r}{r+1}}~\exp({E_{4}}),
\end{align}
where $$H_{i}=\frac{G_{i}-B_{i}A_{i}^{-l}}{(A_{i}-1)G_{i}}, \quad\quad i=1, 2.$$

Since
\begin{align*}
\log A_{1}^{k}&=k\log A_{1}=2k\frac{r+1}{r}\mathscr{P}_{1}+k\frac{r+1}{r}O(\delta),\\
\log A_{2}^{k}&=k\log A_{2}=-2k\frac{r+1}{r}\mathscr{P}_{2}+k\frac{r+1}{r}O(\delta),
\end{align*}
it follows that
\begin{equation}\label{k P1 2}
-kr^{-1}\mathscr{P}_{1}=-\frac{1}{2(r+1)}\log A_{1}^{k}-E_{5,1}, \quad
k\mathscr{P}_{2}=-\frac{r}{2(r+1)}\log A_{2}^{k}-E_{5,2},
\end{equation}
where $$|E_{5,1}|\leq \frac{C}{r}k\delta, \quad |E_{5,2}|\leq C k\delta.$$
We then obtain from \eqref{q_{1, 2k}'} and \eqref{k P1 2} that
\begin{align}\label{q_{1, 2k}''}
&q_{1, 2k}\Big(\mathscr{P}_{1, 2k}-\mathscr{P}_{1, 2k+2}\Big)^{-\frac{1}{r+1}}\Big(\mathscr{P}_{1, 2k+3}-\mathscr{P}_{1, 2k+1}\Big)^{-\frac{r}{r+1}}\nonumber\\
&=q_{1, 2l}\left(\mathscr{P}_{1, 2k}^{2}-\mathscr{P}_{1}^{2}\right)^{-\frac{1}{2(r+1)}}\left(\mathscr{P}_{1, 2k+1}^{2}-\mathscr{P}_{2}^{2}\right)^{-\frac{r}{2(r+1)}}E_{6},
\end{align}
where
$$E_{6}=E_{6,1}E_{6,2}E_{6,3}\exp({E_{4}-E_{5,1}-E_{5,2}}),$$
\begin{align*}
E_{6,1}&:=\left(\frac{\mathscr{P}_{1, 2k}+\mathscr{P}_{1}}{(\mathscr{P}_{1, 2k}-\mathscr{P}_{1})A_{1}^{k}}\right)^{\frac{1}{2(r+1)}}
\left(\frac{-\mathscr{P}_{1, 2k+1}-\mathscr{P}_{2}}{(\mathscr{P}_{2}-\mathscr{P}_{1, 2k+1})A_{2}^{k}}\right)^{\frac{r}{2(r+1)}},\\
E_{6,2}&:=\left(\frac{\mathscr{P}_{1, 2k}-\mathscr{P}_{1}}{\mathscr{P}_{1, 2k+2}-\mathscr{P}_{1}}\right)^{\frac{1}{r+1}}
\left(\frac{\mathscr{P}_{2}-\mathscr{P}_{1, 2k+1}}{\mathscr{P}_{2}-\mathscr{P}_{1, 2k+3}}\right)^{\frac{r}{r+1}},~E_{6,3}:=H_{1}^{\frac{1}{r+1}}H_{2}^{\frac{r}{r+1}}.
\end{align*}

Suppose now that $l=N_{0}(\delta)-1$, $l<k\leq N_{1}(\delta)$, then we have
$$|E_{5,1}|+|E_{5,2}|=\frac{1}{r}O(|\log \delta|^{-1}).$$
We will show that
\begin{equation}\label{4}
E_{6,1}=1+\Big(\frac{r+1}{r}\Big)^{1/2}O(\sqrt{\delta})
\end{equation}
\begin{equation}\label{1}
E_{6,2}=1+O(|\log \delta|^{-1}),
\end{equation}
\begin{equation}\label{2}
q_{1, 2l}E_{6,3}=\frac{r}{r+1}+\frac{r}{r+1}O(|\log \delta|^{-1}),
\end{equation}
and
\begin{equation}\label{3}
|E_{4}|\leq C\left(\frac{r+1}{r}\right)^{2}|\log\delta|^{-1}.
\end{equation}
Once we have these estimates, then \eqref{-y_{1, 2k}+y_{1, 2k+2}} results from \eqref{q_{1, 2k}''}. In the rest, we prove \eqref{4}--\eqref{3}, one by one.

\noindent{\bf Step 1.2. Proofs of \eqref{4}--\eqref{3}.}
To prove \eqref{4}, we obtain from \eqref{p_{1}}, \eqref{2.4}, and \eqref{B_{1}} that
\begin{equation*}
\frac{\mathscr{P}_{1}+\mathscr{P}_{1, 2k}}{\mathscr{P}_{1, 2k}-\mathscr{P}_{1}}A_{1}^{-k}=1+\Big(\frac{r+1}{r}\Big)^{1/2}O(\sqrt{\delta}).
\end{equation*}
Similarly,
$$\frac{\mathscr{P}_{1, 2k+1}+\mathscr{P}_{2}}{(\mathscr{P}_{1, 2k+1}-\mathscr{P}_{2})A_{2}^{k}}
=1+\Big(\frac{r+1}{r}\Big)^{1/2}O(\sqrt{\delta}).$$
\eqref{4} is proved.

To prove \eqref{1}, we first observe that
\begin{equation*}
\frac{\mathscr{P}_{1, 2k}-\mathscr{P}_{1}}{\mathscr{P}_{1, 2k+2}-\mathscr{P}_{1}}=\frac{G_{1}A_{1}^{k+1}-B_{1}}{G_{1}A_{1}^{k}-B_{1}}=A_{1}+\frac{(A_{1}-1)B_{1}}{(A_{1}^{k}-1)B_{1}+\frac{1}{1+\delta-\mathscr{P}_{1}}A_{1}^{k}}.
\end{equation*}
Since $A_{1}>1$, $k\geq |\log\delta|$, and $A_{1}=1+\Big(\frac{r+1}{r}\Big)^{1/2}O(\sqrt{\delta})$, we have
\begin{equation*}
\frac{(A_{1}-1)B_{1}}{(A_{1}^{k}-1)B_{1}+\frac{1}{1+\delta-\mathscr{P}_{1}}A_{1}^{k}}
\leq \frac{A_{1}-1}{A_{1}^{k}-1}=\frac{1}{1+A_{1}+...+A_{1}^{k-1}}\leq\frac{1}{k}=O(|\log \delta|^{-1}).
\end{equation*}
It follows that
\begin{equation*}
\frac{\mathscr{P}_{1, 2k}-\mathscr{P}_{1}}{\mathscr{P}_{1, 2k+2}-\mathscr{P}_{1}}=1+O(|\log \delta|^{-1}).
\end{equation*}
By the similar way,
\begin{equation*}
\frac{\mathscr{P}_{2}-\mathscr{P}_{1, 2k+1}}{\mathscr{P}_{2}-\mathscr{P}_{1, 2k+3}}=1+O(|\log \delta|^{-1}).
\end{equation*}
\eqref{1} is proved.

To prove \eqref{2}, we need to use the following inequality
\begin{equation}\label{ms}
ls-\frac{1}{2}(l-1)ls^{2}\leq 1-(1-s)^{l}\leq ls,\quad\forall~s\in[0, 1].
\end{equation}
Recalling that $l=O(|\log\delta|)$, $\mathscr{P}_{1}=\Big(\frac{r}{r+1}\Big)^{1/2}O(\sqrt{\delta})$, and
$$A_{1}^{-1}=1-2\frac{1+r}{r}\mathscr{P}_{1}+\frac{1+r}{r}O(\delta).$$
Taking $s=2\frac{1+r}{r}\mathscr{P}_{1}+\frac{1+r}{r}O(\delta)$ in \eqref{ms}, we have
\begin{equation}\label{A_{1}^{-l}}
1-A_{1}^{-l}=\frac{1+r}{r}\left(2l\mathscr{P}_{1}+O(|\log\delta|^{2}\delta)\right).
\end{equation}
By using \eqref{q_{1,2k}} with $l=N_{0}(\delta)-1$, we have
$$q_{1, 2l}=\frac{r}{l(r+1)+r}+O(\sqrt{\delta}).$$
We thus have
\begin{align*}
q_{1, 2l}H_{1}&=q_{1, 2l}
\left(\Big(\frac{r}{r+1}\Big)^{1/2}O(\sqrt{\delta})+\Big(1-\Big(\frac{r}{r+1}\Big)^{1/2}O(\sqrt{\delta})\Big)(1-A_{1}^{-l})\right)
\frac{1}{A_{1}-1}\\
&=\frac{r}{r+1}+\frac{r}{r+1}O(|\log \delta|^{-1}),
\end{align*}
here, we used
$$\frac{1}{A_{1}-1}=\frac{r}{r+1}\frac{1}{2\mathscr{P}_{1}}+O(1).$$
Similarly,
$$q_{1, 2l}H_{2}=\frac{r}{r+1}+\frac{r}{r+1}O(|\log \delta|^{-1}).$$
\eqref{2} is proved.

To prove \eqref{3}, we first estimate $E_{2}$. Since $l=O(|\log\delta|)$, we have from \eqref{y_{1,2k}} that
\begin{align*}
\sum_{m=l}^{k-1}\mathscr{P}_{1, 2m}^{2}\leq\sum_{m=l}^{N(\delta)-1}\mathscr{P}_{1, 2m}^{2}+\sum_{m=N(\delta)}^{N_{1}(\delta)-1}\mathscr{P}_{1, 2m}^{2}
\leq& C\left(\frac{1}{l-1}+N_{1}(\delta)\mathscr{P}_{1, 2N(\delta)}^{2}\right)\\
&=O(|\log\delta|^{-1}).
\end{align*}
On the other hand, it follows from \eqref{A_{1}^{-l}} that
\begin{equation}\label{formula A1}
1-A_{1}^{-l}=\Big(\frac{1+r}{r}\Big)^{1/2}O\Big(|\log\delta|\sqrt{\delta}\Big),
\end{equation}
and
\begin{equation*}
\frac{A_{1}^{-l}}{G_{1}-B_{1}A_{1}^{-l}}=\frac{1+\Big(\frac{1+r}{r}\Big)^{1/2}O\Big(|\log\delta|\sqrt{\delta}\Big)}
{G_{1}-B_{1}+B_{1}\Big(\frac{1+r}{r}\Big)^{1/2}O\Big(|\log\delta|\sqrt{\delta}\Big)}
=\frac{r}{r+1}O(|\log\delta|^{-1}).
\end{equation*}
Similarly, we have
\begin{equation*}
\sum_{m=l}^{k-1}\mathscr{P}_{1, 2m+1}^{2}= O(|\log\delta|^{-1}),\quad \frac{A_{2}^{-l}}{G_{2}-B_{2}A_{2}^{-l}}=\frac{r}{r+1}O(|\log\delta|^{-1}).
\end{equation*}
Thus, we infer from \eqref{E_{2}} that
\begin{equation}\label{E_{2}'}
|E_{2}|=C \left(\frac{r+1}{r}\right)^{2}O\Big(|\log\delta|^{-1}\Big).
\end{equation}
By using \eqref{F_{j}}, \eqref{formula A1}, and the fact that $A_{1}>1$, we obtain
\begin{align*}
1\geq f_{1}^{k,l}&\geq\left(\frac{1}{1+\delta-\mathscr{P}_{1}}+B_{1}(1-A_{1}^{-l})\right)\left(\frac{1}{1+\delta-\mathscr{P}_{1}}+B_{1}\right)^{-1}\\
&\geq C\sqrt{\delta}|\log\delta|.
\end{align*}
We then infer from \eqref{E_{3}} that
\begin{equation}\label{E_{3}log F_{1}}
|E_{3,1}\log f_{1}^{k,l}|\leq C\sqrt{\delta}|\log\delta|.
\end{equation}
Similarly,
\begin{align}\label{E_{3}'log F_{2}}
|E_{3,2}\log f_{2}^{k,l}|\leq C\sqrt{\delta}|\log\delta|.
\end{align}

It follows from \eqref{E_{4}}, \eqref{k P1 2}, \eqref{E_{2}'}--\eqref{E_{3}'log F_{2}} that
\begin{align*}
|E_{4}|&\leq \frac{r+1}{r}(k-l)\delta-l\left(\mathscr{P}_{2}-\frac{1}{r}\mathscr{P}_{1}\right)
+|E_{2}|+\frac{1}{r}|E_{3,1}\log f_{1}^{k,l}|+|E_{3,2}\log f_{2}^{k,l}|\\
&\leq C\left(\frac{r+1}{r}\right)^{2}|\log\delta|^{-1},
\end{align*}
and thus \eqref{3} is proved.

\noindent{\bf Step 2. Proof of \eqref{q_{1, 2k}leq e}.} In view of \eqref{2.38}, we have
\begin{equation}\label{est q12k}
q_{1, 2k}\leq q_{1, 2N_{1}(\delta)}\left(\frac{r}{r+\delta+\mathscr{P}_{1}}\right)^{k-N_{1}(\delta)}
\left(\frac{1}{1+\delta-\mathscr{P}_{2}}\right)^{k-N_{1}(\delta)},\quad\forall~ k\geq N_{1}(\delta).
\end{equation}
By using \eqref{q_{1, 2k}q_{1, 2m}}, \eqref{q_{1,2k}}, \eqref{3}, $l=N_{0}(\delta)-1$, and the fact that $\varepsilon$ is sufficiently small, we have
\begin{align*}
&q_{1, 2N_{1}(\delta)}\\
&=q_{1, 2l}\exp\left(N_{1}(\delta)\Big(\mathscr{P}_{2}-\frac{1}{r}\mathscr{P}_{1}\Big)\right){(f_{1}^{N_{1}(\delta),l})}^{\frac{1}{r+1}}{(f_{2}^{N_{1}(\delta),l})}^{\frac{r}{r+1}}\exp(E_{4})\\
&\leq C\Big(\frac{r}{r+1}\frac{1}{|\log\delta|}+O(\sqrt{\delta})\Big)\\
&\quad\cdot\exp\left(N_{1}(\delta)\Big(\mathscr{P}_{2}-\frac{1}{r}\mathscr{P}_{1}\Big)\right){(f_{1}^{N_{1}(\delta),l})}^{\frac{1}{r+1}}{(f_{2}^{N_{1}(\delta),l})}^{\frac{r}{r+1}}\exp(E_{4})\\
&\leq C_{1}\exp\left(-C_{2}\big(\frac{r}{r+1}\big)^{1/2}\frac{1}{\sqrt{\delta}|\log\delta|}\right)
\end{align*}
for some constants $C_{1}$ and $C_{2}$. Coming back to \eqref{est q12k}, we get \eqref{q_{1, 2k}leq e}.
\end{proof}


\end{document}